\RequirePackage{fix-cm}
\documentclass[smallextended]{svjour3}       
\smartqed  

\usepackage{graphicx}
\usepackage[utf8]{inputenc}
\usepackage[T1]{fontenc}
\usepackage{macros}
\usepackage[english]{babel}
\usepackage{amsmath}
\usepackage{tikz-cd}
\usepackage{amssymb}
\usepackage{stmaryrd}
\usepackage{xspace}
\usepackage{setspace}
\usepackage{csquotes}
\usepackage{mathtools}
\usepackage{thmtools}
\usepackage{tdsfrmath}
\usepackage{xfrac}
\usepackage[shortlabels]{enumitem}
\usepackage{geometry}
\usepackage{wasysym}
\usepackage{afterpage}
\usepackage{txfonts}

\spnewtheorem*{nota}{Notation}{\it}{\rm}

\newcommand{\Span}[1]{\operatorname{span}\set{#1}}

\DeclareFontFamily{OMX}{yhex}{}
\DeclareFontShape{OMX}{yhex}{m}{n}{<->yhcmex10}{}
\DeclareSymbolFont{yhlargesymbols}{OMX}{yhex}{m}{n}
\DeclareMathAccent{\wip}{\mathord}{yhlargesymbols}{"F3}
\DeclareMathOperator{\arccosh}{arccosh}
\DeclareMathOperator{\dev}{\mathrm{dev}}
\begin{document}

\title{Parametrisation of decorated Margulis spacetimes using strip deformations
}


\author{Pallavi Panda}


\institute{Pallavi Panda \at
              University of Luxembourg\\
              \email{pallavipanda04@gmail.com} }

\date{Received: date / Accepted: date}

\maketitle

\begin{abstract}
Danciger-Guéritaud-Kassel parametrised the moduli space of Margulis spacetimes, with a fixed convex cocompact linear part. The parametrisation is done by gluing infinitesimal hyperbolic strips along a family of embedded, pairwise disjoint finite geodesic segments of the hyperbolic surface that decompose it into topological disks. We generalise this result to complete finite-area crowned hyperbolic surfaces with their spikes decorated with horoballs, by using hyperbolic as well as parabolic strips. This gives a parametrisation of Margulis spacetimes, decorated with finitely many pairwise disjoint affine light-like lines, called photons.
\keywords{Strip deformations \and arc complex \and Margulis spacetimes}
\end{abstract}

\section{Introduction}
In this paper, we give a parametrisation of the moduli space of a Margulis spacetime decorated by finitely many affine light-like lines. This is done by using strip deformations of the underlying hyperbolic surface. In the first part of this section, we recall the historical context and the recent developments that serve as the motivation and groundwork for our results. In the second part, we summarise our main results.
\label{intro}
\subsection{Historical context}
Bieberbach proved (1910-1912) that any group $\Gamma$ of affine isometries of the $n$-dimensional Euclidean space $\R^n$ that acts properly discontinuously on $\R^n$ contains a finite-index subgroup isomorphic to $\Z^m$, $m\leq n$. Moreover, the quotient $\R^n/\Ga$ is compact if and only if $m=n$. In 1964, Auslander proposed the following conjecture:

	If $\Gamma \subset \operatorname{GL}(n,\R)\ltimes \R^n$ is a finitely generated group that acts on $\R^n$ properly discontinuously and cocompactly, then $\Gamma$ is virtually solvable. 

The conjecture has been proven to be true up to $n=6$. In 1977, Milnor \cite{milnor} asked whether the conjecture remains true if the cocompactness condition is dropped. 
Meanwhile, in 1972, Tits proved that: 

	Let $\Gamma\subset\mathrm{GL}(n, \mathbb F)$ be a finitely generated group, where $\mathbb F$ is a field. Then $\Gamma$ is either virtually solvable or it contains a free group of rank >1. 

Hence the Tits alternative implies that the answer to Milnor's question is negative if and only if there exists a properly discontinuous affine action of a free group (of rank >1).

\paragraph{Margulis spacetimes.} In 1983, Margulis came up with examples for $n=3$. These were complete non-compact Lorentzian manifolds, called \emph{Margulis spacetimes}, obtained as a quotient of the (2,1)-Minkowski space $\Min$ by a free group $\Ga$, acting properly discontinuously by orientation-preserving affine isometries. The group of orientation-preserving affine isometries of $\Min$ is given by $\so\ltimes \R^3$. The special linear group $\so$ is the isometry group of the two-dimensional hyperbolic space $\HP$, which lies in the Minkowski space. Its Lie algebra $\mathfrak{so}_{2,1}$, equipped with its Killing form, is isomorphic to $\Min$. The linear action of $\so$ on $\Min$ coincides with its adjoint action on $\mathfrak{so}_{2,1}$. Consequently, the tangent bundle $\mathrm T(\so)$ is isomorphic to $\so\ltimes \R^3$. We shall denote by $G$ the isomorphic groups $\so,\pgl$ and by $\lalg$, their Lie algebra. Abels-Margulis-Soifer \cite{amssemi}, \cite{amszariski} generalised these examples to higher dimension by showing the existence of properly discontinuous action of non-abelian free Zariski-dense subgroups of $\mathrm{SO}^0(n-1,n)$, with $n$ even. They also showed that such actions cannot exist when $n$ is odd. Margulis introduced an invariant (see Section \ref{mst}), which was later named after him, to detect whether a given action on $\Min$ is proper. This invariant can be interpreted as the first length derivative of a closed geodesic. He proves that (\emph{Opposite Sign Lemma}) the sign of the Margulis invariant must remain constant for a proper action. 
\paragraph{Affine deformations.} Consider the representation $\rho_0:\Ga \hookrightarrow G\ltimes \lalg\simeq \mathrm T(G)$ of a discrete not virtually solvable $\Ga$ acting properly discontinuously on $\Min$, like in the examples of Margulis. Fried and Goldman \cite{fried} proved that by projecting $\Ga$ onto its first coordinate, we virtually get the holonomy representation $\rho: \fg S \rightarrow G$ of a finite-type complete hyperbolic surface $S$. The projection onto the second coordinate $u:\Gamma\rightarrow\lalg$ is a $\rho$-cocycle: for every $\ga,\ga'\in \Ga$, $u$ satisfies $u(\ga\ga')=u(\ga)+\rho(\ga)\cdot u(\ga')$. It gives an infinitesimal deformation of $\rho$. The group $\Gamma$ can thus be written as $\Gamma^{(\rho,u)}:=\{(\rho(\ga), u(\ga)) \mid \ga\in \fg S \}$, which gives an \emph{affine deformation} of $\rho$. 
In the paper \cite{glm}, the authors Goldman, Labourie and Margulis have studied affine deformations of free, discrete subgroups of $G$. An infinitesimal deformation $u$ of $\rho$ is said to be \emph{proper} if $\Gamma^{(\rho,u)}$ acts properly discontinuously on $\Min$.
They \cite{glm} prove that for $\rho$ convex cocompact, the corresponding $u$ is proper if and only if $u$ or $-u$ uniformly lengthens all closed geodesics: 
\begin{equation}\label{ul}
	\inf_{\ga\in \Gamma\smallsetminus \{Id\}} \frac{\mathrm{d}l_{\ga}(\rho)(u)}{l_{\ga}(\rho)} >0
\end{equation} where $l_{\ga}$ is the length function.
In both cases, the set of all such infinitesimal deformations forms an open, convex cone; the cone corresponding to the first case is called the \emph{admissible cone}. This was proved by using a diffused version of the Margulis invariant, which now measures the variations of geodesic currents. 

\paragraph{Strip deformations of compact surfaces.} Danciger-Gu\'eritaud-Kassel \cite{dgk}
study admissible deformations of finite-type hyperbolic surfaces with non-empty boundary and without punctures, using \emph{strip deformations} (Section \ref{sd}), first introduced by Thurston in \cite{thurston}. This is done by cutting the surface along a non-trivial embedded arc, with endpoints on the boundary, and gluing a hyperbolic strip along it. Using the isotopy classes of these arcs, one can construct a simplicial complex called the \emph{arc complex} (Definition \ref{ac}) which depends only on the topology of the surface. Harer \cite{harer} has proved that a certain open dense subset, called the \emph{pruned arc complex} (Definition \ref{pac}) of the arc complex is homeomorphic to an open ball of dimension one less than that of the deformation space. The authors uniquely realised an admissible deformation of the surface by performing strip deformations along families of pairwise disjoint finite arcs, corresponding to a point in the pruned arc complex. 

Drumm \cite{drumm} constructed fundamental domains of some Margulis spacetimes with a convex cocompact linear part using specially crafted piecewise linear surfaces called \emph{crooked planes}. The Crooked Plane Conjecture says that every Margulis spacetime is amenable to such a treatment. Charette, Drumm and Goldman proved this conjecture for rank two free groups in \cite{cdg3}. The general case (with convex cocompact linear part) follows from \cite{dgk} which provides a dictionary between strip deformations and crooked planes. More generally, Smilga \cite{SmilgaFD} gave a construction of fundamental domains for the action of the Abesl-Margulis-Soifer subgroups mentioned above.


\subsection{Main results of the paper}
\paragraph{Surfaces with decorated spikes.}  Crowned surfaces are complete finite-area non-compact surfaces that can be seen as limits of a compact surface with convex polygonal boundary where the vertices (called \emph{spikes}) become ideal.  These were first introduced by Thurston \cite{thurston} as the complement of a geodesic lamination on a closed hyperbolic surface.

Our main aim is to generalise the parametrisation result of Danciger--Gu\'eritaud--Kassel to crowned hyperbolic surfaces whose spikes are decorated with pairwise disjoint horoballs. These decorated surfaces were first introduced by Penner in his study of Decorated Teichmüller Theory \cite{Pennerbordered}, \cite{Pennerpunc}.  A \emph{horoball connection} is a geodesic arc on the surface that joins two decorated spikes. Its length is given by the geodesic segment intercepted by the two horoballs decorating its endpoints.  Penner defined the \emph{lambda length} of a horoball connection and used these lengths to get coordinates to the decorated Teichmuller space. The lambda lengths were one of the primary motivations for the development of the field of cluster algebra. Furthermore, using the arc complex, he gave a cell-decomposition of the decorated Teichmuller space of the surface. 

We define the \emph{admissible cone} of a hyperbolic surface with decorated spikes to be the set of all infinitesimal deformations that uniformly lengthen every horoball connection and every closed geodesic. More precisely, an element $(m,v)$ in the tangent space over a decorated metric $m$ is admissible if and only if it satisfies:
\begin{equation*}\label{ulhc}
	\inf_{\be\in \mathcal{H}}\frac{\mathrm{d}l_{\be}(m)(v)}{l_{\be}(v)} >0,
\end{equation*} where $\mathcal{H}$ is the set of all horoball connections and closed geodesics. 

On the decorated surface, we consider more arcs than in the compact case. In addition to the arcs already mentioned, we allow two new types: finite arcs that are isotopic to a horoball neighbourhood of a spike, and infinite arcs that are embeddings of $[0,\infty)$ such that the finite end is on the boundary and the infinite end converges to a spike. This time the pruned arc complex is defined to be the subspace of the arc complex formed by taking the union of all those simplices $\sigma$ such that the arcs corresponding to the 0-skeleton of $\sigma$ decompose the surface into topological disks with at most one decorated spike. We prove that the pruned arc complex is an open ball for the crowned surfaces and for the decorated surfaces with spikes using Harer's result \cite{harer}.

The strip added along an infinite arc is the region in $\HP$ bounded by two geodesics with the spike as the common endpoint. A strip deformation along a finite arc is defined as in the previous case. 

The main result of this paper gives a parametrisation of the admissible cone of a surface with decorated spikes using the strip deformations:
\begin{theorem} 
	Let $\sh$ be a hyperbolic surface with decorated spikes equipped with a decorated metric $m\in \tei\sh $. Let $\sac \sh$ be its pruned arc complex. Choose $m$-geodesic representatives from the isotopy classes of arcs. Then, the projectivised infinitesimal strip map $\mathbb{P}f:\sac \sh \longrightarrow \ptan \sh$ is a homeomorphism onto its image  $\mathbb{P}^+(\adm m)$, where $\adm m$ denotes the admissible cone over $m$.
\end{theorem}

In \cite{ppstrip}, we proved the above parametrisation result for the decorated ideal polygons, and decorated once-punctured polygons. Their arc complexes are finite, contrary to the bigger surfaces, so some of the methods used in the proofs there are different than those in this paper. 

\paragraph{Decorated Margulis Spacetimes.}
We interpret admissible deformations of surfaces with decorated spikes as Margulis spacetimes with a certain type of decoration by pairwise disjoint lightlike lines (photons), one photon per decorated spike. The photons in the decoration all have the same \emph{handedness} (see Section \ref{hand}) which translates to having the same sign for a quantity for every pair of photons. This can be viewed as analogous to the opposite sign Lemma of Margulis. In this context, the above theorem provides fundamental domains of the Margulis spacetimes, adapted to the photons. We prove the following theorem:

\begin{theorem}\label{dms}
	Let $\sh$ be a hyperbolic surface with decorated spikes and let $\rho:\fg {\sh} \rightarrow G$ be a holonomy representation.
	Let $\mst$ be the space of all decorated Margulis spacetimes with convex cocompact linear part as $\rho$. Then there is a natural bijection $\Psi: \sac\sh \rightarrow\mst$ induced by the projectivised strip map.
\end{theorem}

The paper is structured into sections in the following way:
Sections  \ref{prelim}-\ref{surfspike} recapitulate the necessary vocabulary from hyperbolic, Lorentzian and projective geometry, and introduces every type of surface mentioned above along with their deformation spaces and admissible cones. In Section \ref{arc}, we discuss the arcs and the arc complexes of the different types of surfaces and study their topology. Section \ref{sd} gives the definitions of the various strip deformations along different types arcs and some estimations that will be required in the proofs. We also give a recap of the main steps of the proof of their main result in \cite{dgk}. Section \ref{main} contains the proof of our main parametrisation theorem for surfaces with decorated spikes. Finally, in Section \ref{dmst} we introduce decorated Margulis spacetimes and give the proof of Theorem \ref{dms}.

\section{Preliminaries}\label{prelim}
In this section we recall the necessary vocabulary and notions and also prove some results in hyperbolic geometry that will be used in the rest of the paper.
\subsection{Minkowski space $\Min$}
\begin{definition}
	The \emph{Minkowski space} $\R^{2,1}$ is the affine space $\R^3$ endowed with the quadratic form $\norm \cdot$ of signature $(2,1)$: \[\text{for } v=(x_1,x_2,x_3)\in \R^3,\quad \norm v ^2=x_1^2+x_2^2-x_3^2.\] 
	
\end{definition}
There is the following classification of points in the Minkowski space: a non-zero vector $\mathbf{v}\in \Min$ is said to be
\begin{itemize}
	\item \emph{space-like} if and only if $\norm {\mathbf v}^2>0$,
	\item \emph{light-like} if and only if $\norm {\mathbf v}^2=0$,
	\item \emph{time-like} if and only if $\norm {\mathbf v}^2<0$.
\end{itemize}
A vector $\mathbf v$ is said to be \emph{causal} if it is time-like or light-like. A causal vector $\mathbf v=(x,y,z)$ is called \emph{positive} (resp.\ \emph{negative}) if $z>0$ (resp.\ $z<0$). Note that by definition of the norm, every causal vector is either positive or negative. 
The set of all light-like points forms the \emph{light-cone}, denoted by $$L:=\{\mathbf v=(x,y,z)\in \Min \mid x^2+y^2-z^2=0\}.$$
The \emph{positive} (resp.\ \emph{negative}) cone is defined as the set of all positive (resp.\ \emph{negative}) light-like vectors.
\paragraph{Subspaces.} A vector subspace $W$ of $\Min$ is said to be 
\begin{itemize}
	\item \emph{space-like} if $W\cap C=\{(0,0,0)\}$,
	\item \emph{light-like} if $W\cap C=\Span {\mathbf v}$ where $\mathbf v$ is light-like,
	\item \emph{time-like} if $W$ contains at least one time-like vector.
\end{itemize}
A subspace of dimension one is going to be called a line and a subspace of dimension two a plane. The adjective "affine" will be added before the words "line" and "plane" when we are referring to some affine subspace of the corresponding dimension.
\paragraph{Duals.} Given a vector $\mathbf {v}\in\Min$, its dual with respect to the bilinear form of $\Min$ is denoted $\bf v^{\perp}$. For a light-like vector $\bf v$, the dual is given by the light-like hyperplane tangent to $C$ along $\Span{\bf v}$. For a space-like vector $\bf v$, the dual is given by the time-like plane that intersects $C$ along two light-like lines, respectively generated by two light-like vectors $\bf v_1$ and $\bf v_2$ such that $\Span {\bf v}=\bf v_1^{\perp}\cap \bf v_2^{\perp}$. Finally, the dual of a time-like vector $\bf v$ is given by a space-like plane. One way to construct it is to take two time-like planes $W_1,W_2$ passing through $\bf v$. Then the space $\bf v^\perp$ is the vectorial plane containing the space-like lines $W_1^\perp$ and $W_2^\perp$.

\subsection{The different models of the hyperbolic 2-space}
In this section we recall some vocabulary and introduce notations related to the different models for the hyperbolic plane, that will be used in the calculations and proofs later.
\paragraph{Hyperboloid model.} The classical hyperbolic space of dimension two $\HP$ can be identified with the upper sheet of the two-sheeted hyperboloid  $\{\mathbf v=(x,y,z)\in\Min \mid \norm{\mathbf v}^2=-1\},$ along with the restriction of the bilinear form. It is the unique (up to isometry) complete simply-connected Riemannian 2-manifold of constant curvature equal to -1. Its isometry group is isomorphic to $\so$ and the identity component $\sop$ of this group forms the group of its orientation-preserving isometries; they preserve each of the two sheets of the hyperboloid individually. If the hyperbolic distance between two points $\mathbf u,\mathbf v\in \HP$ is denoted by $d_{\HP}(\mathbf u,\mathbf v)$, then $\cosh d_{\HP}(\mathbf u,\mathbf v)= -\bil {\mathbf u }{\mathbf v}$. The geodesics of this model are given by the intersections of time-like hyperplanes with $\HP$.

\paragraph{Klein's disk model.} This model is the projectivisation of the hyperboloid model. Let $\mathbb{P}:\Min\smallsetminus \{\mathbf 0\} \longrightarrow \pp$ be the projectivisation of the Minkowski space. The projective plane $\pp$ can be considered as the set $A \cup \mathbb{RP}^1$, where $A:=\set{(x,y,1)\,|\,x,y\in \R}$ is an affine chart and the one-dimensional projective space represents the line at infinity, denoted by $\pli$. The $\p$-image of a point $\mathbf v\in \Min$ is denoted by $[\mathbf v]$. A line in $A$, denoted by $\pl$, is defined as $A\cap V$ where $V$ is a two-dimensional vector subspace of $\Min$, not parallel to $A$. 

In the affine chart $A$, the light cone is mapped to the unit circle and the hyperboloid is embedded onto its interior.  This is the Klein model of the hyperbolic plane; its boundary a circle. This model is non-conformal. The geodesics are given by open finite Euclidean straight line segments, denoted by $l$, lying inside $\HP$, such that the endpoints of the closed segment $\pls$ lie on $\HPb$. 
The distance metric is given by the Hilbert metric 
$d_{\HP}(w_1,w_2)=\frac{1}{2}\log [p,w_1;w_2,q]$, where $p$ and $q$ are the endpoints of $\pls$, $l$ being the unique hyperbolic geodesic passing through $w_1,w_2\in \HP$, and the cross-ratio $[a,b;c,d]$ is defined as $\frac{(c-a)(d-b)}{(b-a)(d-c)}$. The group of orientation-preserving isometries is identified with $\mathrm{PSU}(1,1)$. A point $p$ is called \emph{real} (\emph{ideal}, \emph{hyperideal}) if $p\in \HP$ (resp. $p\in \HPb$, $p\in \pl\cup A\backslash \cHP$).

The dual of $\pli$ is the point $(0,0,1)$ in $A$.
The dual of any other projective line $\pl=A\cap V$ is given by the point $A\cap V^{\perp}$. The dual $p^{\perp}$ of a point $p\in \pp$ is the projective line $A\cap {\Span p}^{\perp}$. If $l$ is a hyperbolic geodesic, then $l^{\perp}$ is defined to be $\pl^{\perp}$; it is given by the intersection point in $\pp$ of the two tangents to $\HPb$ at the endpoints of $\pls$.

\emph{Notation:} We shall use the symbol $\cdot^{\perp}$ for referring to the duals of both linear subspaces as well as their projectivisations.

\paragraph{Upper Half-plane Model.} The subset $\{z=x+iy\in \C \,| \,y>0\}$ of the complex plane is the upper half-space model of the hyperbolic space of dimension 2.
The geodesics are given by semi-circles whose centres lie on $\R$ or straight lines that are perpendicular to $\R$. We shall call the former as \emph{horizontal} and the latter as $vertical$ geodesics. The boundary at infinity $\HPb$ is given by $\R \cup \{\infty\}$. The orientation-preserving isometry group is given by $\psl$ that acts by Möbius transformations on $\HP$. 

Recall that we denote by $G$ the isomorphic groups $\mathrm{Isom}(\HP), \so,\pgl$ and by $\lalg$ the Lie algebra of $G$.

An \emph{open horoball} $h$ based at $p\in\HPb$ is the projective image of $H (\mathbf v)=\{\mathbf w\in \HP\mid \bil{\mathbf w}{\mathbf{v}}>-1 \}$ where $\mathbf v$ is a future-pointing light-like point in $\pinv p$. 
If $k\geq k'>0$, then $H(k\mathbf v_0)\subset H(k'\mathbf v_0)$.

The boundary of an open horoball $h(p)\subset \HP$ based at $p\in \HPb$ is called a \emph{horocycle}. It is the projective image of the set $$h(\mathbf v):=\{\mathbf w\in \HP \mid \ang{\mathbf w,\mathbf v}=-1 \}.$$
In the projective disk model, it is a Euclidean ellipse inside $\HP$, tangent to $\HPb$ at $p$. In the upper half-plane model, horocycles are either Euclidean circles tangent to a point on the real line or horizontal lines which are horocycles based at $\infty$. In the Poincaré disk model, a horocycle is an Euclidean circle tangent to $\HPb$ at $[p]$. A geodesic, one of whose endpoints is the centre of a horocycle, intersects the horocycles perpendicularly. Note that any horoball is completely determined by a future-pointing light-like vector in $\Min$ and vice-versa. From now onwards, we shall use either of the notations introduced above to denote a horoball. Finally, the set of all horoballs of $\HP$ forms an open cone (the positive light cone). 

Given an ideal point $p\in \HPb$, a \emph{decoration} of $p$ is the specification of an open horoball centred at $p$. A geodesic, whose endpoints are decorated, is called a \emph{horoball connection}.
The following definition is due to Penner \cite{penner}.

\begin{definition}\label{horolength}
	The length of a horoball connection joining two horoballs $\mathbf v_1,\mathbf v_2$ is given by
	\begin{equation*}
		l:= \ln(-\frac{  \ang{\mathbf v_1,\mathbf v_2}}{2}).
	\end{equation*}
\end{definition}

It is the signed length of the geodesic segment intercepted by the corresponding horocycles. In particular, is the horoballs are not disjoint, then the length of the horoball connection is negative. 

\subsection{Killing Vector Fields of $\HP$}
The Minkowski space $\Min$ is isomorphic to $(\lalg, \kappa)$ where $\lalg$ is the Lie algebra of $G=\pgl$ and $\kappa$ is its Killing form, via the following map:
\[\mathbf v=(x,y,z)\mapsto V=\begin{pmatrix}
	y & x+z\\
	x-z & -y
\end{pmatrix} .\]

The Lie algebra $\lalg$ is also isomorphic to the set $\mathscr{X}$ of all Killing vector fields  of $\HP$:
\[V\mapsto \bra{ 
	\begin{array}{ccc}
		X_v:&\HP\longrightarrow &\mathrm T\HP\\
		&\mathbf p\mapsto &\frac{d}{dt} (e^{tV}\cdot \mathbf p)|_{t=0}
\end{array}}\]
Next, one can identify $\Min$ with $\mathscr{X}$ via the map:
\[ \bf v\mapsto \bra{
	\begin{array}{ccc}
		X_v:&\HP\longrightarrow &\mathrm T\HP\\
		&p\mapsto &\bf {v} \mcp \mathbf p
\end{array}}
\] where $\mcp$ is the Minkowski cross product:
\[(x_1,y_1,z_1)\mcp(x_2,y_2,z_2):=(-y_1z_2+z_1y_2,-z_1x_2+x_1z_2, x_1y_2-y_1x_2).\]

\begin{property}
	
	Using these isomorphisms, we have that 
	\begin{itemize}
		\item A spacelike vector $\mathbf v$ corresponds, in $\mathscr{X}$, to an infinitesimal hyperbolic translation whose axis is given by $\mathbf v^{\perp}\cap \HP$. If $\mathbf v^+$ and $\mathbf v^-$ are respectively its attracting and repelling fixed points in $C^+$, then $(\bf v^-, v,v^+)$ are positively oriented in $\Min$.
		\item A lightlike vector $\mathbf v$ corresponds, in $\mathscr{X}$, to an infinitesimal parabolic element that fixes the light-like line $\Span {\mathbf v}$.
		\item A timelike vector $\mathbf v$ corresponds, in $\mathscr{X}$, to an infinitesimal rotation of $\HP$ that fixes the point $\frac{\mathbf v}{\sqrt {-\norm{\mathbf v}}}$ in $\HP$.
	\end{itemize} 
\end{property}

\begin{property}\label{tang}~
	\begin{enumerate}
		\item Given a light-like vector $\mathbf v\in\Min$, the set of all Killing vector fields that fix $\Span {\mathbf v}$ is given by its dual $\mathbf v^{\perp}$. In $\pp$, the set of projectivised Killing vector fields that fix $[\mathbf v] \in \HPb$ is given by the tangent line at $[\mathbf v]$.
		\item The set of all Killing vector fields that fix a given ideal point $p\in\HPb$ and a horocycle in $\HP$ with centre at $p$ is given by $\Span {\mathbf v}$, where $\mathbf v\in \mathbb{P}^{-1}(p)$ in $\Min$.
		\item The set of all Killing vector fields that fix a given hyperbolic geodesic $l$ in $\HP$ is given by $\mathbb{P}^{-1}(l^{\perp})$.
	\end{enumerate}
	
\end{property}
\subsection{Convex cocompact surfaces}\label{ccc}
Any orientable compact surface is of the form $S_{g,n}:=\mathbb{S}^2\#((\mathbb{T}^2)^{\#g})\#((\mathbb{D}^2)^{\# n})$ where
\begin{itemize}
	\item $\s 2$ is a sphere of dimension 2,
	\item $\mathbb{T}^2$ is the topological surface of $\R^2/\Z^2$,
	\item $\mathbb{D}$ is a closed 1-disk,
	\item  the variable $g\in \N$ is called the genus of the surface and is additive under the connected sum, i.e, $S_g \# S_{g'}=S_{g+g'}$. 
	\item the variable $n\in \N$ denotes that number of boundary components.
\end{itemize}
Next, we shall look at some examples and their common names.
\begin{example}
	When $n=0$, the surface is called \emph{closed}.
\end{example}
\begin{example}
	Suppose that $g=0$. 	
	\begin{enumerate}
		\item When $n=1$, we get back the disk $\mathbb D$.  
		\item When $n=2$, we get an \emph{annulus}.
		\item When $n=3$, the surface is called a \emph{pair of pants}.
	\end{enumerate}
\end{example}

\begin{example}
	When $g=1,n=1$, we shall call the surface a \emph{one-holed} torus.
\end{example}

The Euler characteristic of such a surface is given by $\chi(S_{g,n})=2-2g-n$.

\subsection{Non-orientable surfaces}
Any compact non-orientable surface is of the form $T_{h,n}=(\pp)^{\# h}\#((\mathbb{D}^2)^{\# n})$ where 
\begin{itemize}
	\item $\pp$ is the projective plane,
	\item the variable $h\in \N$ here is again additive under the connected sum; the surface corresponding to $h=0$ is the 2-sphere, which is orientable; so when we write $T_{h,n}$, we implicitly assume that $h>0$. Also, we have the equality $T_h\#S_g=T_{h+2g}$, for any $h>0$.  
\end{itemize}
\begin{example}
	When $h=1, n=1$, we get the \emph {Möbius strip}.
\end{example}
\begin{example}
	When $h=2, n=0$, we get the \emph{Klein's bottle}.
\end{example}

We are primarily interested in those compact surfaces $S$ which are hyperbolic and have non-empty boundary. From the Uniformisation Theorem, we know that the Euler characteristic of such a surface, denoted by $\chi(S)$, is negative. The following is the list of all the connected orientable and non-orientable surfaces that aren't hyperbolic, and hence excluded from the discussion:
\[  
\begin{array}{l@{\quad}l}
	S_{0,0}:\text{ a sphere }\s 2, &S_{0,2}: \text{ annulus,}\\
	S_{1,0}:\text{ a torus $\mathbb{T}^2$},&T_{1,1}: \text{ a closed M\"{o}bius Strip}\\
	T_{1,0}: \text{a projective plane $\pp$},&T_{2,0}: \text{Klein's bottle.}
\end{array}
\]

A complete finite-area hyperbolic metric with totally geodesic boundary on a compact hyperbolic surface $S_c:=S_{g,n}$ or $T_{h,n}$ ($n>0$) is given by the following information:
\begin{itemize}
	\item A discrete faithful representation, called a holonomy representation 
	\[ 
	\rho:\fg{S} \longrightarrow \pgl,
	\]
	that maps each boundary component $b_i$ to a hyperbolic element. When $S_c=S_{g,n}$, the image $\rho(\fg {S_c})$ is a Fuchsian subgroup of $\psl$.
	\item A developing map $\dev: \wt{S_c}\longrightarrow \HP$, such that the following diagram commutes: for all $\ga \in \fg{S_c}$
	\[  
	\begin{tikzcd}
		\wt{S_c}\ar[d,"\ga"']\ar[r,"\dev "]&\HP\ar[d,"\rho(\ga)"]\\
		\wt{S_c}\ar[r,"\dev"']&\HP
	\end{tikzcd}
	\] for all $\ga\in\fg{S}$.
	Here, $\wt{S_c}$ is the universal cover of $S_c$, on which an element $\ga\in \fg{S_c}$ acts by deck transformations. 
\end{itemize}

It follows from these conditions that the group $\Gamma:=\rho(\fg {S_c})$ is a discrete finitely generated free group of $\pgl$ containing only hyperbolic elements. The $\dev$-image is a simply-connected region in $\HP$ bounded by infinite geodesics corresponding to the lifts of its boundary components $\partial_i S$, for every  $i=1,\ldots,n$. These geodesics are pairwise disjoint in $\ol{\HP}$.

The \emph{deformation space} $\tei {S_c}$ of the surface is the set of conjugacy classes of all possible holonomy representations. It is a connected component of the set $$\set{[\rho]: \rho \text{ is discrete, faithful}; \forall i, \rho(\partial_i {S_c} ) \text{ is hyperbolic}}\subset\mathrm{Hom}(\fg {S_c}, G)/G,$$ where the action of $G$ is by conjugation.

Let $S_c$ be a compact hyperbolic surface endowed with a metric $m=[\rho]\in \tei {S_c}$. Given an element $[\ga ]\in \fg {S_c}\smallsetminus \{Id\}$, there exists a unique closed $m$-geodesic in this homotopy class, denoted by $\ga_g$.  
\begin{definition}\label{length}
	The \emph{length function} is defined in the following way:
	\[  
	\begin{array}{rrcl}
		l_\ga:&\tei {S_c}&\to& \R_{>0}\\
		&[\rho]&\mapsto & 2\arccosh\pa{\frac{\tr{\rho(\ga_g)}}{2}}.
	\end{array}
	\]
\end{definition}

The following is a well-known result (for e.g. see \cite{flp}) which is usually proved using Fenchel-Nielson coordinates:
\begin{theorem}\label{defc}
	Let ${S_c}$ be a compact hyperbolic surface with geodesic boundary. 
	\begin{enumerate}
		\item If $S_c=S_{g,n}$, then its deformation space $\tei {S_{g,n}}$ is homeomorphic to an open ball of dimension $6g-6+3n$.
		\item If $S_c=T_{h,n}$, then its deformation space $\tei {T_{h,n}}$ is homeomorphic to an open ball of dimension $3h-6+3n$.
	\end{enumerate}
\end{theorem}

Next we recall infinitesimal deformations and cocycles.
\begin{definition}
	An infinitesimal deformation of a metric $m\in \tei {S_c}$ is a vector of the tangent space $\tang {S_c}$.
\end{definition}

Let $S_c=S_{g,n}$ or $T_{h,n}$ be a compact surface with non-empty boundary equipped with a hyperbolic structure as above. An \emph{infinitesimal deformation} of its holonomy representation $\rho:\fg {S_c}\longrightarrow G$ is a vector of $T_{\rho}\mathrm{Hom}(\fg {S}, G)$. It can be seen as an equivalence class of smooth paths $\{\rho_t \}_{t\in\R}$ with $\rho_0=\rho$. Given such a path, we have that
\begin{align*}
	\ddt{\rho_t}\Bigr\rvert_{t = 0}(\fg{S_c})\subset TG\simeq G\ltimes \lalg,
\end{align*} 
in other words,  for every $\ga\in \fg S_c$, $\ddt{\rho_t}\Bigr\rvert_{t = 0}(\ga)= (\rho(\ga),u(\ga)),$

The map $u:\fg{S_c}\longrightarrow \lalg$ defined above satisfies the \emph{cocycle} condition:
\begin{align}\label{cocycle}
	\text{for every $\ga_1,\ga_2\in \fg{S_c}$, }& u(\ga_1\ga_2)=u(\ga_1)+\mathrm{Ad}(\rho(\ga_1))\cdot u(\ga_2).
\end{align}

A map $u:\fg{S_c}\longrightarrow \soal$, satisfying \eqref{cocycle} is called a $\rho$-cocycle. 

\begin{definition}
	A \emph{$\rho$-coboundary} is a $\rho$-cocycle $u$ such that for some $v_0\in \lalg$
	\begin{align}
		u(\ga)=\mathrm{Ad}(\rho(\ga))v_0-v_0,&\text{ for every $\ga\in \fg{S_c}$}.
	\end{align}  
\end{definition}
Two $\rho$-cocycles are equivalent if they differ by a coboundary. The set of equivalence classes of all $\rho$-cocycles forms the first cohomology group $\mathrm{H}^1_{\rho}(\fg {S_c},\lalg)$. An element $[u]$ of this group is an infinitesimal deformation of the metric $[\rho]$, \ie $[u]\in \tang {S_c}$.

Next, we will define a specific type of infinitesimal deformation of a compact surface, known as an admissible deformation.
\begin{definition}
	Let $S_c$ be a compact (possibly non-orientable) hyperbolic surface with non-empty boundary. Let $m\in \tei {S_c}$ and $v\in \tang {S_c}$. Then $v$ is said to be an \emph{admissible} deformation of $m$ if it satisfies:
	\begin{align}\label{adm}
		\inf_{\ga\in \Gamma\smallsetminus \{Id\}} \frac{\mathrm{d}l_{\ga}(m)(v)}{l_{\ga}(m)} >0,
	\end{align} where $l_{\ga}$ is the length function as introduced in Definition \eqref{length}.
\end{definition}
In other words, an infinitesimal deformation is admissible if and only if the length of every non-trivial closed loop of $S_c$ is uniformly lengthened. The following theorem was proved by Goldman-Labourie-Margulis \cite{glm}:
\begin{theorem}\label{glm}
	The set of all admissible deformations of a compact hyperbolic surface $S_c$ with non-empty totally geodesic boundary forms an open convex cone of $\tang {S_c}$.
\end{theorem}

\subsection{Margulis spacetimes}\label{mst}
\begin{definition}Let $\rho_0:\Ga \hookrightarrow G\ltimes \lalg$ be the representation of a discrete not virtually solvable group $\Ga$ acting properly discontinuously and freely on $\Min$.
	Then the quotient manifold $M:=\Min/\rho_0(\Ga)$ is called a \emph{Margulis spacetime}. 
\end{definition}

As mentioned in the introduction, Fried and Goldman \cite{fried} proved that by projecting $\Ga$ onto its first coordinate, we get the holonomy representation $\rho: \fg S \rightarrow G$ of a finite-type complete hyperbolic surface $S$. The projection onto the second coordinate $u:\Gamma\rightarrow\lalg$ is a $\rho$-cocycle, which is also an infinitesimal deformation of $\rho$.
The group $\Gamma$ can thus be written as $\Gamma^{(\rho,u)}:=\{(\rho(\ga), u(\ga)) \mid \ga\in \fg S \}$, which gives an \emph{affine deformation} of $\rho$. 
Goldman--Labourie--Margulis proved in \cite{glm} that for $\rho$ convex cocompact, then the group $\Gamma^{(\rho,u)}$ acts properly if the $\rho$-cycle $u$ or $-u$ uniformly lengthens all closed geodesics of the hyperbolic surface, \ie the equivalence class $[u]$ of $u$, modulo coboundaries, lies in the admissible cone $\adm {[\rho]}$ of the hyperbolic surface $\HP/\rho(\fg S)$.

Let $\rho$ be a convex cocompact representation as before and $u$ be a $\rho$-cocycle. Margulis defined an invariant that is used to detect the properness of such a cocycle. For every non-trivial $\ga\in \fg S$, its image $\rho(\ga)$ is a hyperbolic element of $G$ with eigenvalues of the form $\lambda, 1, \lambda^{-1}$. Let $v_1,v_2$ be two future-pointing light-like eigenvectors corresponding to the eigenvalues $\lambda, \lambda^{-1}$, and let $v_0(\ga)$ be the eigenvector of unit norm with eigenvalue 1 such that $(v_1,v_0,v_2)$ is positively oriented. Then the \emph{Margulis invariant} is defined as the map:
\[ 
\begin{array}{rrcl}
	\al_u:&\fg S&\longrightarrow &\R\\
	& \ga &\mapsto &\ang {u(\ga),v_0(\ga)}.
\end{array}
\]
The map $\al_u$ depends only on the cohomology class of $u$. Margulis showed the following lemma about the properness of a cocycle and the sign of the invariant:
\begin{lemma}[Opposite sign lemma, Margulis \cite{margulis}]
	Suppose that $\Gamma^{(\rho,u)}\subset \mathrm{Isom}^+(\Min)$ acts properly on $\Min$. Then either for every $\ga\in \fg S$, $\al_u(\ga)>0$ or for every $\ga\in \fg S$, $\al_u(\ga)<0$.
\end{lemma}

Next, we  recall \emph{crooked planes}.

Take any space-like vector $\mathbf v\in \Min$. Then the associated Killing field is hyperbolic and has an attracting and a repelling fixed point at $p_+,p_-\in \HPb$, respectively. Their preimages are light-like lines: $\mathbb{P}^{-1} p_+=\R \mathbf{v_+}$, $\mathbb{P}^{-1} p_-=\R \mathbf{v_-}$, where $\mathbf{v_+},\mathbf{v_-}$ are future-pointing light-like vectors. Denote by ${l_{\mathbf {v}}}$, the oriented hyperbolic geodesic from endpoints $p_-$ to $p_+$. It divides $\HP$ into two half-spaces — the one lying to the right of ${l_{\mathbf {v}}}$ is denoted by $H_+(\mathbf v)$, the one to the left is denoted by $H_-(\mathbf v)$. The geodesic ${l_{\mathbf {v}}}$ is transversely oriented in the following way: a directed geodesic $l$ in $\HP$ transverse to ${l_{\mathbf {v}}}$ is said to be pointing in the positive direction if the point $[p_+]\in \HP$ lies to its left. When we refer to the geodesic ${l_{\mathbf {v}}}$ along with its transverse orientation, we shall denote it by $\vec{{l_{\mathbf {v}}}}$.
\begin{definition}
	A \emph{left crooked plane $\mathcal P(\mathbf v)$ centered at 0, directed by a space-like vector $\mathbf v$} is a subset of $\Min$ that is the union of the following sets:
	\begin{itemize}
		\item A \emph{stem}, defined as $\stem:=\{\mathbf w\in \Min\mid \norm {\mathbf w}^2\leq 0\}\cap \du {\mathbf v}$. It meets the light-cone along the two light-like lines $\R \mathbf{v_+}$, $\R \mathbf{v_-}$.
		\item Two \emph{wings}: The connected component of $\du {\mathbf{v_+}}\smallsetminus \R \mathbf{v_+}$ (resp.\ of $\du {\mathbf{v_-}}\backslash \R \mathbf{v_-}$), that contains all the hyperbolic Killing fields whose attracting fixed point is given by $p_+$ (resp.\ $p_-$), is called a \emph{positive wing} (resp.\ a \emph{negative wing}). They are denoted by $\mathcal{W}^+(\mathbf v)$ and $\mathcal{W}^-(\mathbf v)$, respectively.
	\end{itemize}
\end{definition} 
For any vector $\mathbf {v_0}\in \Min$, the subset $\mathcal{P}(\mathbf {v_0},\mathbf v):=\mathbf{v_0}+\mathcal P(\mathbf v)$ is an \emph{affine} left crooked plane \emph{centered at $\mathbf{v_0}$} and directed by a space-like vector $\mathbf v$. Then, $\mathcal{P}(\mathbf 0,\mathbf v)=\mathcal P(\mathbf v)$.

\paragraph{Crooked Halfspaces:} The connected component of $\Min\backslash \mathcal{P}(\mathbf{v_0},\mathbf v)$ containing the Killing fields whose non-repelling fixed points (space-like, time-like, light-like) lie in the half-plane $H_+(\mathbf v)\subset\HP$ (resp.\ $H_-(\mathbf v)$) is called the \emph{positive crooked half-space} (resp.\ \emph{negative crooked half-space)}, denoted by $\mathcal{H}^+(\mathbf v)$ (resp.\ $\mathcal{H}^-(\mathbf v)$).

Next we recall the definition of stem quadrant of a transversely oriented hyperbolic geodesic, as defined in \cite{bcdg}.

Let $\mathbf v\in \Min$ be a space-like vector, $\mathbf{v_+},\mathbf{ v_-}$ be future-pointing light-like vectors in $\du{\mathbf v}$ and $\vect{ {l_{\mathbf {v}}}}$ be the hyperbolic geodesic in $\HP$ with endpoints at $[\mathbf{v_+}], [\mathbf{v_-}]$ and oriented towards $[\mathbf{v_+}]$. 
\begin{definition}
	The set $\sq:=\R_{>0}\mathbf{v_+}-\R_{>0}\mathbf{v_-}$ is called the \emph{stem quadrant} of the transversely oriented geodesic ${l_{\mathbf {v}}}$, associated to the positively oriented triplet $(\mathbf{v_+},\mathbf v, \mathbf{v_-})$.
\end{definition}


\section{Crowned surfaces and their decorations}\label{surfspike}
In this section we introduce crowned surfaces and give a construction of hyperbolic metrics.  For that first we need to define crowned hyperbolic surfaces. We start with the description of the simplest surface of this type and then gradually increase the topological complexity to obtain more generic examples. 

\paragraph{Ideal Polygons.} Let $D^q$ be a closed 2-disk with $q\ (\geq 0)$ points removed from its boundary. When $q \geq3$, we can put a complete finite-area hyperbolic metric on $D^q$ by taking the convex hull in $\HP$ of $q$ distinct points on $\HPb$. These points are called \emph{vertices}, and the \emph{edges} are the infinite geodesics of $\HP$ joining two consecutive vertices. Figure \ref{idealpolygon} shows an ideal pentagon in the projective model of $\HP$.
\begin{figure}[h!]
	\centering
	\frame{	\includegraphics[width=8cm]{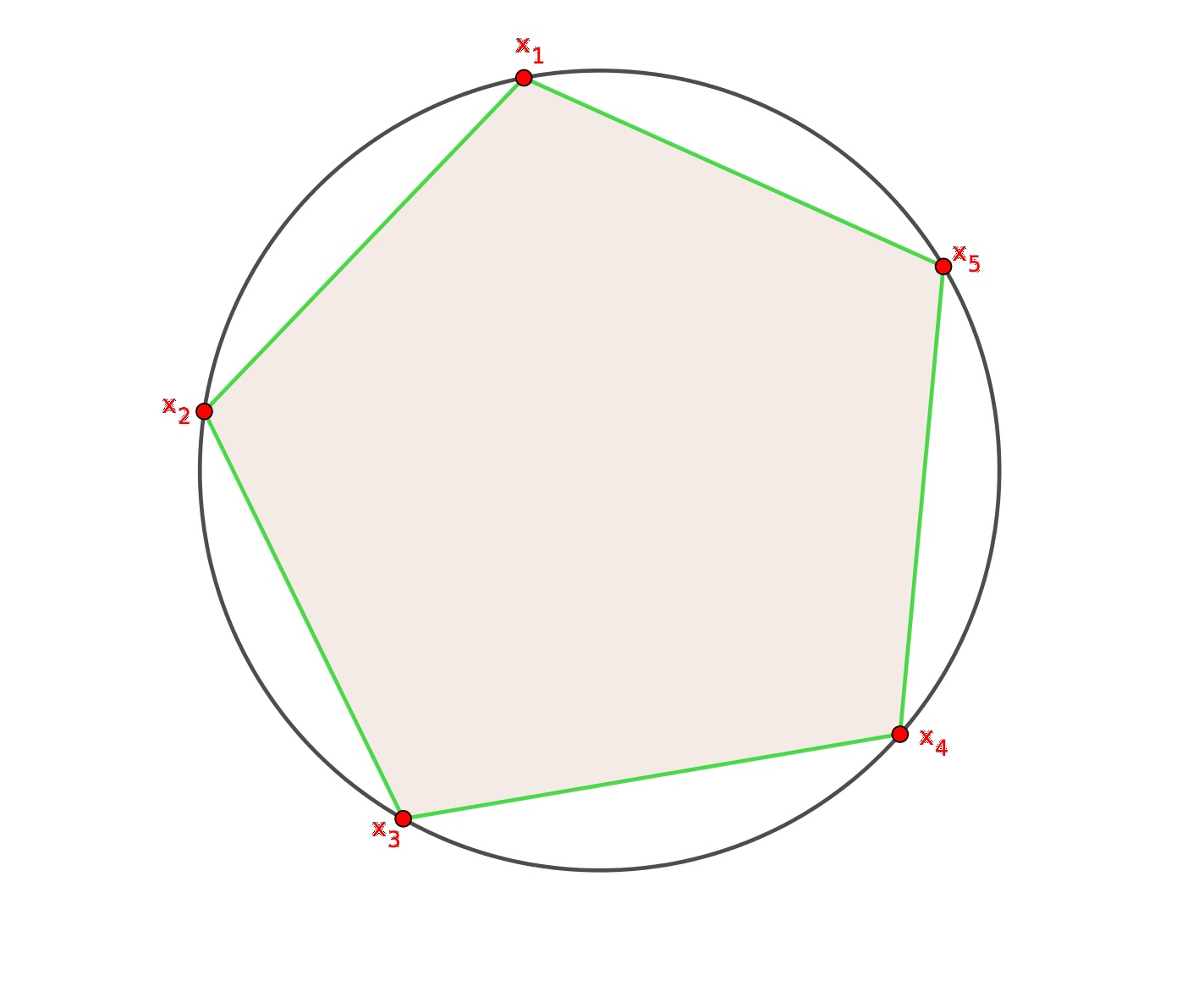}}
	\caption{An ideal pentagon}
	\label{idealpolygon}
\end{figure}

\begin{definition}
	Let $S_c$ be $S_{g,m}$ or $T_{h,m}$, with $m\geq 0$. Consider $k\ (>0)$ disks $D^{q_1},\ldots, D^{q_k}$, with $q_j\geq1$. 
	Then, the surface $\spike$ obtained by taking the connected sum $S_c\#D^{q_1}\#\ldots\#D^{q_k}$ is called a \emph{crowned surface}.
\end{definition}
Let $n:=m+k$ and $Q:=\sum\limits_{i=1}^k q_i$.
Given an orientable (resp.\ non-orientable) crowned surface such that $6g-6+3n+Q>0$ (resp.\ $3h-6+3n+Q>0$), we can put a complete finite-area hyperbolic metric on it. Firstly we give a construction of the hyperbolic metric in the case of "small" hyperbolic surfaces:


\begin{figure}[h!]
	\centering
	\frame{\includegraphics[width=7cm]{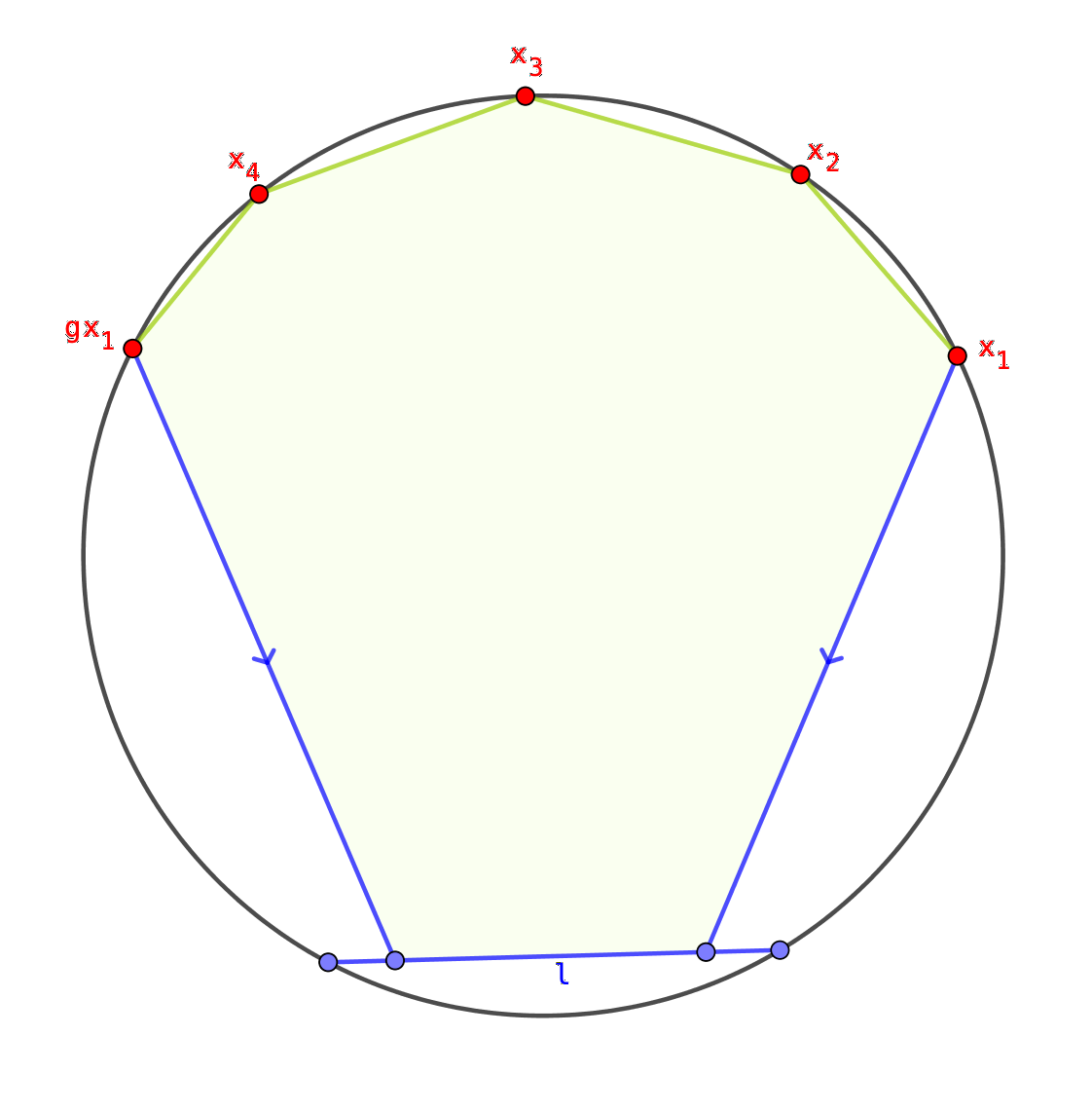}
		\includegraphics[width=8cm]{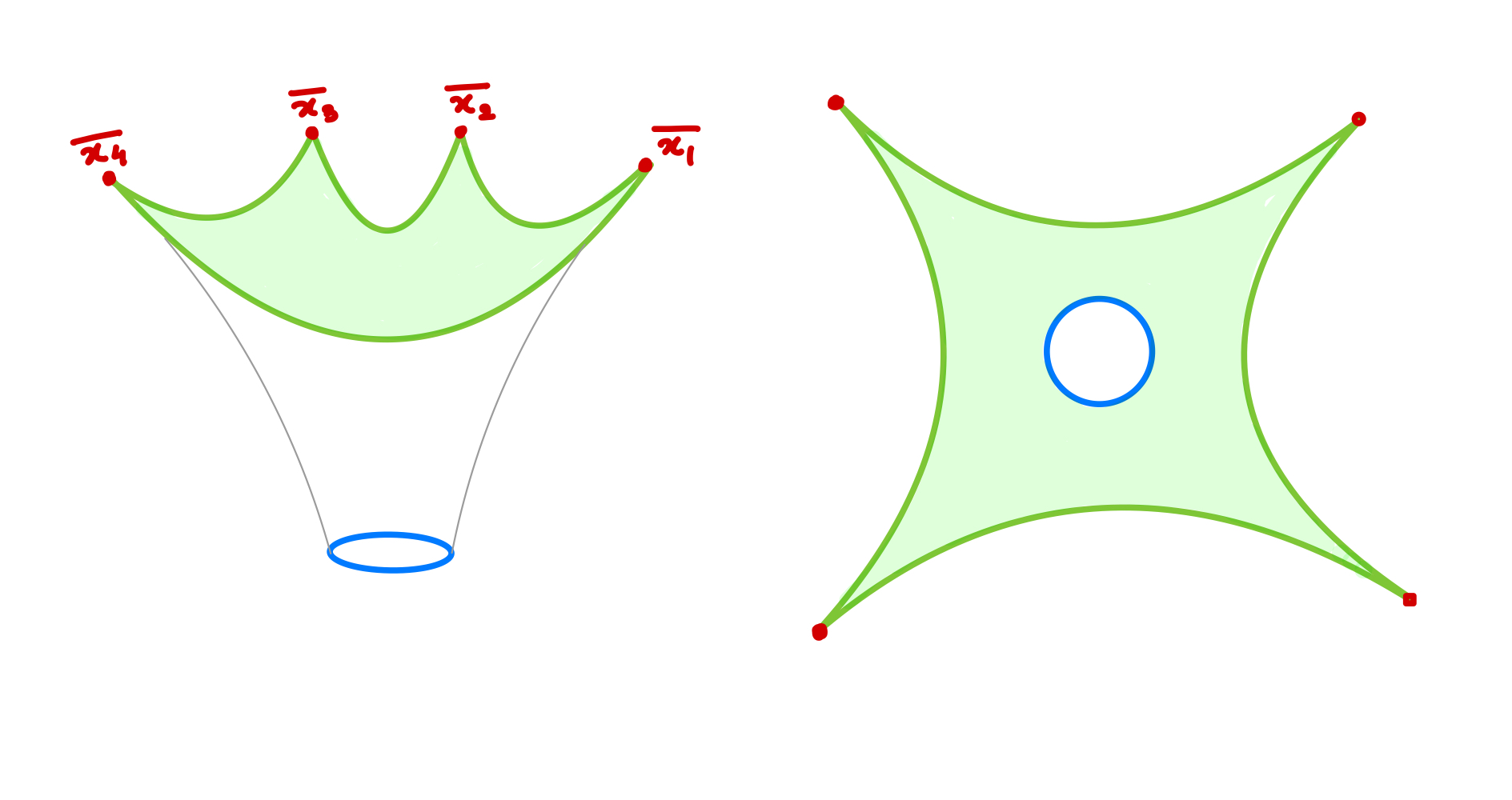}}
	\caption{A fundamental domain for an ideal one-holed square}
	\label{holed}
\end{figure}
\begin{example}
	The orientable surface $\hd^2 \# D^{q}$, for $q>0$, is called a \emph{crown} and is denoted by $\holed q$. Its boundary consists of one simple closed curve, denoted by $\ga$, and $q$ open intervals. Its fundamental group $\fg{\holed q}$ is generated by a loop $b$ in the free  homotopy class of $\ga$ and is isomorphic to $\Z$. Next, we put a hyperbolic structure on $\holed q$ in the following way:
	
	Let $\ga\in \psl$ be a hyperbolic element whose axis is a bi-infinite geodesic, denoted by $l$. See Fig.\ \ref{holed}. It divides the boundary circle $\HPb$ into two open intervals. Choose a point $x_1$ in any one of them and take $(q-1)$ distinct points $x_2,\ldots,x_q$ on the same interval between $x_1$ and its image $\ga\cdot x_1$. Mark all the points of their $\ang{\ga}$-orbit. All of them lie on the same side of $l$ as the initial points. Join consecutive pairs using infinite geodesics. Drop two perpendiculars from $x_1$ and $\ga\cdot x_1$ to $l$ and identify them using $\ga$. The quotient surface is a complete finite-area hyperbolic surface with geodesic boundary. If $\rho:\fg {\holed n}\longrightarrow \psl$ is the holonomy representation, then $\rho([b])=g$. The images of the ideal points $x_1,\ldots,x_q$ in the quotient are called \emph{spikes}.
\end{example}

\begin{example}
	The orientable surface $S:=\s 2\# D^{q_1}\# D^{q_2}$ for $q_1,q_2>0$ is called a ($q_1,q_2$)-\emph{spiked annulus.} Any connected component of its boundary is homeomorphic to an open interval. It contains exactly one isotopy class $[\ga]$ of (non-trivial) simple curves.
	Its fundamental group is again isomorphic to $\Z$. 
	\begin{figure}[h!]
		\centering
		\frame{	\includegraphics[width=10cm]{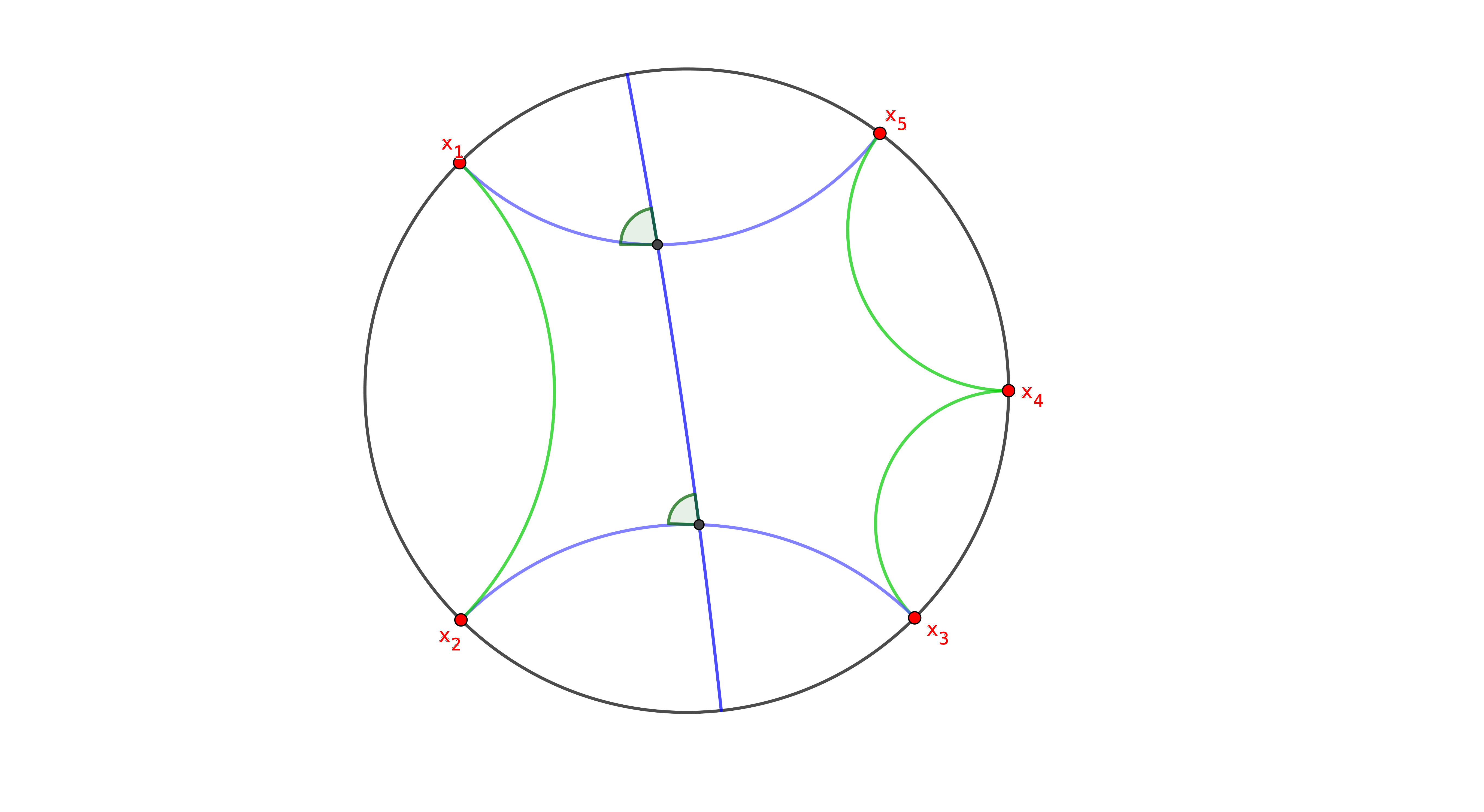}}
		\caption{A fundamental domain for a (1,2)-spiked annulus}
		\label{spann}
	\end{figure}
	Take an ideal $q$-gon in $\HP$, where $q=q_1+q_2+2.$ Its vertices are denoted by $x_1,\ldots,x_{q}$ in the anti-clockwise direction. Let $l_1,l_2$ be the edges joining the pairs of vertices ($x_{q_1+1},x_{q_1+2}$) and $(x_{q},x_1)$, respectively. Let $g\in\psl$ be a hyperbolic isometry whose axis intersects both $l_1$ and $l_2$ at the same angle, and  whose translation length is given by the distance between the two points of intersection. Then the quotient surface $S=\ip q /\sim$, where for every $z\in l_1$, $z\sim g\cdot z$, is a complete finite-area hyperbolic surface with geodesic boundary, homeomorphic to a ($q_1,q_2$)-spiked annulus. Its holonomy representation $\rho:\fg {S}\longrightarrow \psl$ maps the generator $[\ga]$ to $g$.
\end{example}
\begin{figure}[h!]
	\centering
	\frame{\includegraphics[width=10cm]{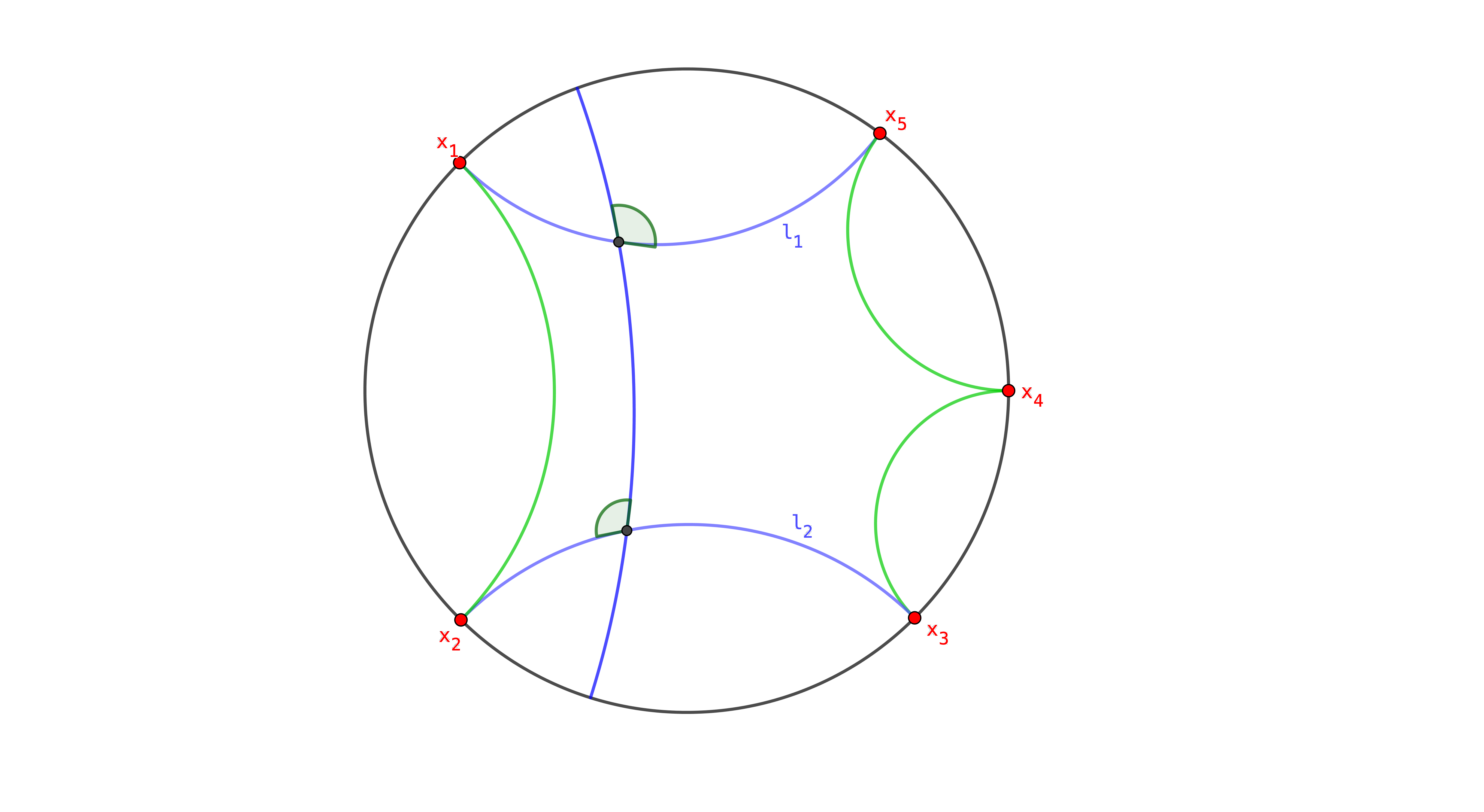}}
	\caption{A fundamental domain for a 3-spiked Möbius strip}
	\label{spmob}
\end{figure}
\begin{example}
	The non-orientable surface $\pp\# D^{q}$, for $q>0$, is called a {spiked M\"obius strip}. Its orientation double cover is a ($q,q$)-spiked annulus. Similar to the previous example, we consider an ideal $q+2$-gon with marked vertices $x_1,\ldots,x_{q+2}$. See Fig.\ \ref{spmob}. Take any two distinct edges $l_1,l_2$ of the polygon. 
	Let $l$ be a geodesic that intersects $l_1$ and $l_2$ such that the angles of intersections are complementary, and let $d$ be the distance between 
	$l\cap l_1$ and $l\cap l_2$. Let $h\in \pgl$ be the gliding reflection along $l$ with translation length $d$. Then the quotient $\ip n/\sim $, where for every $z\in l_1$, $z \sim h\cdot z$, is a complete finite-area crowned hyperbolic surface.
\end{example}

 A simple closed loop on $\spike=S_c\#D^{q_1}\#\ldots\#D^{q_k}$ is called a \emph{peripheral loop} if it is either freely homotopic to a boundary curve of the compact surface $S_c$ or separates a crown from the rest of the surface $\spike$. 

\begin{nota}
	For every $i=1,\ldots,n:=k+m$,
	let $q_i \geq 1$ be the number of spikes of the $i$-th crown if $1\leq i \leq k$, and $q_i=0$ if $k<i\leq m$ (i.e.\ if the $i$-th peripheral loop is isotopic to a boundary component of $S_c$). 
	 Define the spike vector $\vec{q}:=(q_1,\ldots,q_n)$ and let $Q$ be the total numberof spikes. Finally, an orientable (resp.\ non-orientable) surface with genus $g$ (resp.\ $h$), $n$ peripheral loops and spike vector $\vec{q}$ is denoted by $\gensurf$ (resp.\ $\gensurfn$).
\end{nota}
\begin{nota}
	Now we extend the above notation for the exceptional cases that do not have hyperbolic closures. For ideal $q$-gons, ($q\geq 3$), we shall use the notation $S_{0,1}^{(q)}$. For a  ($q_1,q_2$)-spiked annulus, we shall use $S_{0,2}^{(q_1,q_2)}$. Finally, for a M\"obius strip with $q\,  (\geq 1)$ spikes, we shall use $T_{1,1}^{(q)}$.
\end{nota}


Next we 
parametrize all marked, complete, finite-area hyperbolic metrics
on a generic crowned surface $\spike$. Let $\partial_1,\ldots,\partial_n$ be the peripheral loops of this surface, where $n=m+k$. We shall assume that the closure $S_c$ is hyperbolic since we have already treated the cases where it is not hyperbolic. Choose a base point $p\in S_c$. For $i=1,\ldots, n$, let $b_i$ be a loop in $S_c$ based at $p$ that is freely homotopic to $\partial_i$. Start with a hyperbolic metric on the surface $S_c$ with totally geodesic boundary. Let $\rho:\fg{S_c,\ p}\rightarrow \pgl$ be a holonomy representation. Let $\ga_i:=\rho([b_i]) \in \pgl $.
\begin{figure}[!ht]
	\centering
	\frame{\includegraphics[width=12cm]{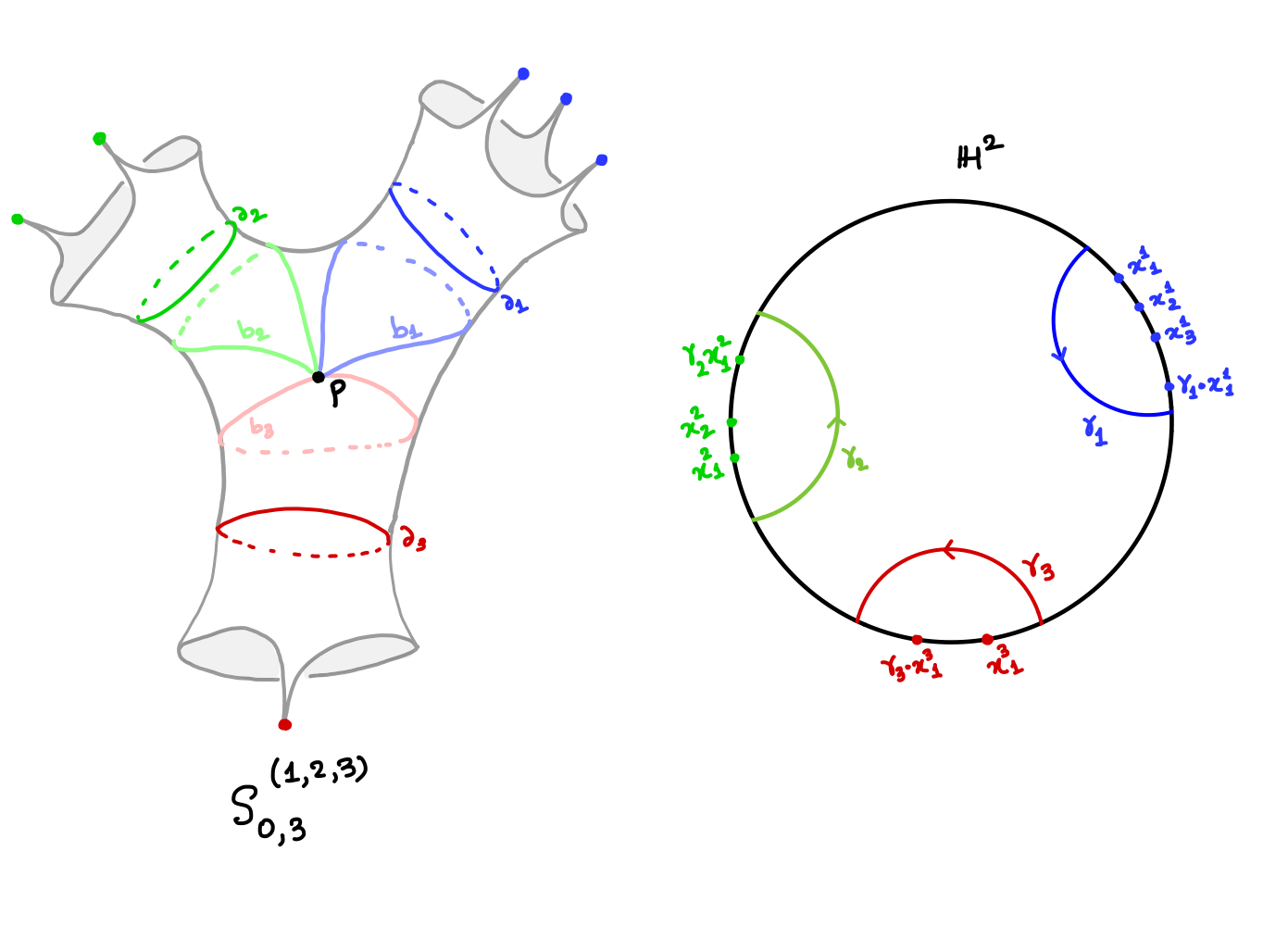}}
	\caption{A metric on $S^{1,2,3}_{0,3}$}
	\label{ucspikepants}
\end{figure}

Now we will choose the ideal points representing the spikes on $\spike$. See Fig.\ref{ucspikepants}. In the universal cover, we look at the oriented axes of $\ga_i$. They separate the lift of $S_c$ from a halfplane. We choose $Q$ distinct points on $\HPb$ in the following way— whenever $q_i>0$, take $q_i$ ordered ideal points $\ul{x^i}=(x^i_1,\ldots,x^i_{q_i})$ on the same side of the axis of $\ga_i$ as the halfplane, such that $x^i_{q_i}$ lies between $x^i_{q_{i}-1}$ and $\ga_i\cdot x^i_{1}$. Denote by $\mathbf x$ the tuples of ideal points $\ul{x^1},\ldots,\ul{x^k}$.

A 
(marked)
metric on a surface with spikes $\spike$ can be seen as an ordered pair $(\rho,\mathbf{x})$.  Two pairs $(\rho,\mathbf{x}), \, (\rho',\mathbf{x'})$ are said to be equivalent if there exists an element $g\in \pgl$ such that for all $\ga\in \fg \spike$, $\rho'(\ga)=g\rho(\ga)g^{-1}$ and $\mathbf{x'}=g\cdot \mathbf{x}$. The set of all the equivalence classes forms the deformation space of a crowned surface $\spike$, denoted by $\tei\spike.$

The following theorem about the dimension of the deformation space of a surface with non-decorated spikes is analogous to Theorem \ref{defc}.
\begin{theorem}\label{defsp}
	Let $\spike$ be a crowned hyperbolic surface with $Q$ spikes. 
	\begin{enumerate}
		\item If $\spike=\gensurf$, then its deformation space $\tei \spike$ is homeomorphic to an open ball of dimension $6g-6+3n+Q$.
		\item If $\spike=\gensurfn$, then its deformation space $\tei \spike$ is homeomorphic to an open ball of dimension $3h-6+3n+Q$.
	\end{enumerate}
\end{theorem}
\begin{definition}
	The smallest closed geodesically convex subset of a crowned hyperbolic surface $\spike$ which contains all of its closed geodesics is called the \emph{convex core} of $\spike$.
\end{definition} 	
Given a crowned hyperbolic surface $\spike=\gensurf$ or $\gensurfn$, its convex core is a compact surface homeomorphic to $S_{g,n}$ or $T_{h,n}$. 
The following is the list of all crowned surfaces whose convex cores are not hyperbolic:
\begin{itemize}
\item Ideal polygons and ideal punctured polygons have empty convex cores.
\item The convex cores of a spiked M\"obius strip, of a crown and of a spiked annulus are all homeomorphic to a circle.
\end{itemize}

Now we define a decorated hyperbolic surface.

\begin{definition}
	A hyperbolic surface with decorated spikes is obtained from a crowned hyperbolic surface of type $\gensurf$ or $\gensurfn$ by decorating each spike with a horoball. Such a surface is denoted by $\decogs$ (when orientable) and $\decogsn$ (when non-orientable).
\end{definition}

We shall use the symbol $\sh$ for referring to both cases at once.

The deformation space $\tei {\sh}$ of such a surface is the trivial $\R_{>0}^{Q}$ bundle over $\tei{\spike}$ where the fibre over a point in the base space is given by the Busemann functions of the horoballs based at the $Q$ spikes of the surface $\spike$. A point in $\tei\sh$ is called a \emph{decorated metric} and is given by the equivalence class $[\rho,\mathbf{x}, \mathbf{h}]$, where the vector $\mathbf{h}=(h_1,\ldots,h_Q)$ records the horoball at every spike.

A \emph{horoball connection} on a surface $\sh$ is a geodesic path joining two not necessarily distinct decorated spikes.
It is the image in the quotient of a horoball connection (see Definition \ref{horolength}) in $\HP$ joining a pair of decorated ideal vertices in $\HPb$. 

Recall that a decorated ideal vertex corresponds to a unique light-like vector. 
Let $(p_1,h_1)$ and $(p_2,h_2)$ be lifts of the two decorated vertices that are joined by the horoball connection $\beta$. Let $v_1,v_2$ be the corresponding light-like vectors. Then we have the following length function for horoball connections:
\[ 
\begin{array}{rrcl}
	l_{\be}: &\tei {\sh}&\longrightarrow &\R\\
	& m &\mapsto&\ln(-\frac{  \ang{\mathbf v_1,\mathbf v_2}}{2}).
\end{array}
\]
The set of all horoball connections is denoted by $\mathcal{H}$.
Using Theorem \ref{defsp}, we have that 
\begin{theorem}\label{defh}
	Let $\sh$ be a hyperbolic surface with $Q$ decorated spikes. 
	\begin{enumerate}
		\item If $\sh=\decogs$, then its deformation space $\tei \decogs$ is homeomorphic to an open ball of dimension $6g-6+3n+2Q$.
		\item If $\sh=\decogsn$, then its deformation space $\tei \decogsn$ is homeomorphic to an open ball of dimension $3h-6+3n+2Q$.
	\end{enumerate}
\end{theorem}

Next we define admissible deformations for hyperbolic surfaces with decorated spikes.

\begin{definition}
	Let $\sh$ be a hyperbolic surface with decorated spikes. The \emph{admissible cone} for a given metric $m\in \tei {\sh}$, denoted by $\adm m$, is the set of all infinitesimal deformations of $m$ that uniformly lengthen every horoball connection and every closed curve. 
\end{definition}
\begin{remark}
	Lengthening every simple  horoball connection lengthens every two-sided curve in the interior of the surface, and every peripheral loop bounding a crown. But one-sided curves and spike-less boundary components are not lengthened. 
\end{remark}
If the decoration on $\sh$ is such that the closures of the decorating horoballs are all pairwise disjoint, then an element $v\in \tang {\sh}$ is admissible if and only if it satisfies the following condition: 
\begin{align}\label{hcon}
	\inf_{\be\in \mathcal{H}}\frac{\mathrm{d}l_{\be}(m)(v)}{l_{\be}(m)} >0,
\end{align}
where $l_{\be}$ is the length function associated to horoball connection $\beta$ (Definition \eqref{horolength}) and $\mathcal{H}$ is the set of all horoball connections.
If some of the decorating horoballs of the spikes of $\sh$ overlap then an admissible $v$ satisfies
\[ 
\inf_{\be\in \mathcal{H}^-}\mathrm{d}l_{\be}(m)(v)>0,
\] where $\mathcal H^-$ is the set of horoball connections with non-positive length, and \eqref{hcon} for horoball connections with positive length.
\begin{remark}
An admissible deformation of $(\sh,[\rho,\mathbf{x},\mathbf{h}])$ determines uniquely an admissible deformation of $(\sh,[\rho,\mathbf{x},\lambda\mathbf{h}])$, for any $\lambda\in \R^*$. 
\end{remark}
In Section \ref{sd}, we prove that the admissibe cone is an open convex cone.

\section{Arcs and the arc complex}\label{arc}
\subsection{Definitions}
In this section we recall the arcs and the arc complexes of crowned surfaces as well as surfaces with decorated spikes. These have been discussed in detail in \cite{ppstrip} (Section 3). We restate the definitions here for the sake of completude.

An \emph{arc} on a hyperbolic surface $S$ with non-empty boundary, possibly with spikes, is an embedding $\al$ of a closed interval $I\subset \R$ into $S$ such that exactly one of the following holds:
\begin{enumerate}
	\item If $I=[a,b]$, then the endpoints  verify $\al(a),\al(b) \in \partial S$, and $\al(I)\cap S=\set{\al(a),\al(b)}$. 
	\item If $I=[a,\infty)$ then the finite endpoint satisfies $\al(a)\in \partial S$ and the infinite end converges to a spike, \ie $\al(t)\overset{t\to \infty}{\longrightarrow} x$, where $x$ is a spike.
\end{enumerate}  

An arc $\al$ of a hyperbolic surface $S$ with non-empty boundary is called \emph{non-trivial} if each connected component of $S\smallsetminus \{\al \}$ has at least one spike or generalised vertex. 

Let $\mathscr A$ be the set of all non-trivial arcs of the two above types.
\begin{definition}\label{ac}
	The \emph{arc complex} of a surface $S$, generated by a subset $\mathcal{K}\subset \mathscr A$, is a simplicial complex $\ac S$ whose base set $\ac{S}^{(0)}$ consists of the isotopy classes of arcs in $\mathcal K$, and there is an $k$-simplex for every $(k+1)$-tuple of pairwise disjoint and distinct isotopy classes. 
\end{definition}

The elements of $\mathcal{K}$ are called \emph{permitted} arcs and the elements of $\mathscr A\smallsetminus\mathcal{K} $ are called \emph{rejected} arcs. 
Next we specify the elements of $\mathcal{K}$ for the different types of surfaces:
\begin{itemize}
	\item In the case of a hyperbolic surface possibly with undecorated spikes, the set $\mathcal K$ of permitted arcs comprises of non-trivial finite arcs that separate at least two spikes from the surface.
	\item In the case of a hyperbolic surface with decorated spikes, the set $\mathcal K$ of permitted arcs comprises of non-trivial finite arcs that separate at least one spike, and infinite arcs of type 2 whose infinite ends converge to spikes and whose finite ends lie on the boundary of the surface. 
\end{itemize}

\begin{remark}
	\begin{enumerate}
		\item Two isotopy classes of arcs of $S$ are said to be disjoint if it is possible to find a representative arc from each of the classes such that they are disjoint in $S$. Such a configuration can be realised by geodesic segments in the context of hyperbolic surfaces. 
		In our discussion, we shall always choose such arcs as representatives of the isotopy classes. 
	\end{enumerate}
\end{remark}

The 0-skeleton $\sigma^{(0)}$ of a top-dimensional simplex $\sigma$ of the arc complex $\ac S$ is called a \emph{triangulation} of the surface $S$.

\begin{definition}\label{filling} We define a \emph{filling} simplex of the arc complex of the different types of surfaces:
	\begin{itemize}
		\item For a hyperbolic surface possibly with non-decorated spikes, a simplex $\sigma$ is said to be filling if the arcs corresponding to $\sigma^{(0)}$  decompose the surface into topological disks with at most two spikes. 
		\item For a surface with decorated spikes, a simplex $\sigma$ is said to be filling if the arcs corresponding to $\sigma^{(0)}$  decompose the surface into topological disks with at most one spike.
	\end{itemize}
\end{definition}
From the definition it follows that any simplex containing a filling simplex is also filling. 
\begin{definition}\label{pac}
	The \emph{pruned arc complex} of a surface $S$, denoted by $\sac S$ is the union of the interiors of the filling simplices of the arc complex $\ac S$.
\end{definition}
Every point $x\in \sac {S}$ is contained in the interior of a unique simplex, denoted by $\sigma_x$, \ie there is a unique family of arcs $\{\al_1,\ldots,\al_p\}$, namely the 0-skeleton of $\sigma_x$, such that \[ x=\sum_{i=1}^p t_{i} \al_i, \, \sum_{i=1}^p t_i =1,\,\text{ and } \forall i, \, t_i>0 .\]
Define the \emph{support} of a point $x\in \sac S$ as $\supp x:= \sigma_x^{(0)}$.

In this section we shall be studying the topology of the pruned arc complexes of hyperbolic surfaces with undecorated spikes, followed by hyperbolic surfaces with decorated spikes. 

\subsection{The arc complex of a crowned or compact surface}
Firstly, we shall discuss a theorem by Harer which proves that the pruned arc complex of an orientable surface is an open ball. In his paper, the terminology used is different from what we have seen up until now so we shall give a quick introduction to the objects involved in his result. Then by interpreting his result in the appropriate manner, we shall prove that the pruned arc complex in the case of orientable surfaces with spikes is an open ball. Finally, we shall derive the same result for non-orientable surfaces.
\paragraph{Harer's Terminology:}Let $S_{g,r,s}$ be an orientable surface of genus $g$ with $r$ boundary components and $s$ punctures:
\[  S_{g,r,s}:= S_{g,r}\smallsetminus \{y_1,\ldots,y_s\},\] where $y_1,\ldots,y_s$ are points in the interior of $S_{g,r}$, that play the role of spikeless boundary components in our case. Mark $Q$ distinct points $x_1,\ldots,x_Q$ on the boundary $\partial S_{g,r,s}$ such that each boundary component, denoted by $\partial_i S_{g,r,s}$ for $i=1,\ldots,r$, contains at least one such point. These points shall play the role of spikes. Let $\Omega=\{x_1,\ldots, x_Q,y_1,\ldots y_s\}$.

The deformation space $\tei {S_{g,r,s}}$ is an open ball of dimension $N_0:=6g-6+3r+2s$.  Define $\mathcal{T}(\Omega):=\{(m,\lambda)\}$, where $m\in \tei {S_{g,r,s}}$ and $\lambda$ is a positive projective weight on the points of $\Omega$. Then, $$\mathcal{T}(\Omega)= {\tei {S_{g,r,s}}} \times \ball{s+Q-1}\simeq \ball {6g-7+3r+3s+Q}.$$ 

Consider $\mathcal{K}$ to be the set of embedded arcs in $S_{g,r}$ whose endpoints belong to $\Omega$. The arc complex spanned by the arcs in $\mathcal{K}$ is denoted by $\ac {S_{g,r,s}}$. A simplex $\sigma$ of the arc complex is said to be "big" if the arcs corresponding to its 0-skeleton divide the surface $S_{g,r,s}$ into topological disks. The pruned arc complex $\sac {S_{g,r,s}} $ is defined to be the union of the interior of the big simplices. 

Finally, let $\mcg {S_{g,r,s}}$ be the mapping class group of the surface whose elements fix the points in $\Omega$. Then, Harer \cite{harer} proves the following theorem:
\begin{theorem}\label{harer}
	There is a natural homeomorphism $\Phi:\mathcal{T}(\Omega)\longrightarrow \sac {S_{g,r,s}}$ that commutes with the action of the mapping class group $\mcg {S_{g,r,s}}$.
\end{theorem}
\paragraph{Interpretation:} Let $\gensurf$ be an orientable hyperbolic surface with undecorated spikes with $k$ boundary components homeomorphic to a circle. Recall that in the case of hyperbolic surfaces with undecorated spikes, the permitted arcs are finite and have both their endpoints on two (not necessarily distinct) connected components of the boundary of the surface. By comparing these arcs with the permitted arcs used by Harer, we get that punctures and spikeless boundaries play the same role in the two surfaces; similarly, the points $y_1,\ldots,y_Q$ can be interpreted as spikes in our case. Then, $k=s$ and $n=r+s$. So, the arc complex $\ac \gensurf$ is isomorphic to the arc complex $\ac {S_{g,r,s}}$. Also, the definition of a big simplex is the same in the two approaches. Hence, from Theorem \eqref{harer} we have that

\begin{corollary}\label{pacundeco}
	The pruned arc complex $\sac\gensurf$ of a connected orientable surface $\gensurf$ with non-decorated spikes is homeomorphic to an open ball of dimension $6g-7+3n+Q$.
\end{corollary}

Next, we prove a similar result for non-orientable surfaces possibly with spikes using Harer's theorem. 

The \emph{orientation covering} of a surface $S$ is defined as the pair ($\ol S, \pi$), where,
\[ \ol S:=\{(p, o(p) \mid p\in S,  o (p) \text{ is an orientation of } T_pS) \} \] and 
\[\begin{array}{cccc}
	\pi:& \ol\gensurfn&\longrightarrow &\gensurfn\\
	&(p,o(p))&\mapsto&p
\end{array} .
\] Since the tangent space $T_pS$ has exactly two orientations, $\pi$ is a two-sheeted covering map. Since $\gensurfn$ is non-orientable, one has that $\ol \gensurfn$ is a orientable surface with $$\chi(\ol \gensurfn)=2\chi(\gensurfn)<0.$$ So we get that $\overline{\gensurfn}$ is hyperbolic and that it is of the form $S_{h-1,2n}^{\vec{q}\sqcup\vec{q}}$, where $\vec{q}\sqcup\vec{q}:=(q_1,\ldots,q_n,q_1,\ldots,q_n)$.
Let  $\Upsilon:S_{h-1,2n}^{\vec{q}\sqcup\vec{q}}\longrightarrow S_{h-1,2n}^{\vec{q}\sqcup\vec{q}}$ be the covering automorphism that exchanges the two points in every fibre of $\pi$.  

We shall revert back to Harer's notation — suppose that $\gensurfn$ has $k$ spikeless boundary components. Then, $S_{h-1,2n}^{2\vec{q}}$ has $s:=2k$ spikeless boundary components. Let $r:=2(n-k)$. We shall be working with $S_{h-1,r,s}$ which is an orientable surface of genus $h-1$ with $r$ boundary components and $s$ punctures. From Theorem \eqref{harer}, we know that $\sac{S_{h-1,r,s}}$ is an open ball of dimension $${6(h-1)-7+3(r+s)+2Q={6h-13+6n+2Q}}.$$
Let $\mathcal{I}(\ac{S_{h-1,r,s}})$ be the subset of $\ac{S_{h-1,r,s}}$ that is invariant under the action of $\Upsilon$. Since the homeomorphism $\Phi:\mathcal{T}(\Omega)\longrightarrow \sac{S_{h-1,r,s}} $ commutes with the action of $\Upsilon$, we get that $\Phi(\mathcal{I}(\ac{S_{h-1,r,s}}))$ is the set of points $(m,\lambda)\in \mathcal{T}(\Omega)$ such that $m$ and $\lambda$ are invariant under $\Upsilon$. Every permitted arc $\al$ of $\gensurfn$ lifts to two disjoint arcs $\al_1,\al_2$ in $S_{h-1,r,s}$ because $\Upsilon$ is a double cover and an arc is simply-connected. These two arcs are interchanged by the action of $\Upsilon$. So, the isotopy classes $[\al^1],[\al^2]$ as well as the 1-simplex generated by them belong to $\mathcal{I}(\ac{S_{h-1,r,s}})$. Consequently, we get the following map between the arc complexes:
\[ \begin{array}{rccl}
	h:&\ac \gensurfn&\longrightarrow&\mathcal{I}(\ac{S_{h-1,r,s}})\\
	(0-\text{skeleton})&[\al]&\mapsto& \frac{[\al^1]+[\al^2]}{2},\\
	(k-\text{skeleton})&\sum_{i=1}^{k+1}t_i[\al_i]&\mapsto&\sum_{i=1}^{k+1}t_i\frac{[\al_i^1]+[\al_i^2]}{2},
\end{array} \]
where $k\leq N_0$, $t_i\geq0$ for $i=1,\ldots,k+1$ ,with $\sum_{i=1}^{k+1}t_i=1$.

\begin{lemma}
	The map $h:\sac \gensurfn\longrightarrow\mathcal{I}(\sac{S_{h-1,r,s}})$ is an isomorphism.
\end{lemma}
\begin{proof}
	Firstly, we show that this map is well-defined. A point $x\in \sac {S}$ belongs to the interior of a unique big simplex $\sigma_x$ of $\ac S$: 
	\[  x=\sum_{i=1}^{N_0} t_i\, [\al_i]\text{, with }t_i\in (0,1), [\al_i]\in \sigma_x^{(0)}, \text{ for every }i=1,\ldots,N_0\text{ and }\sum_i t_i=1.\]
	The union $\bigcup\limits_i e_i$ of arcs decomposes the surface $S$ into topological disks with at most two spikes. Being simply connected, they lift to twice as many disks partitioning the double cover. So the simplex formed by $\{[\al_i^1],[\al_i^2]\}_i$ is big. Hence we get that $h(x)\in \sac{S_{h-1,r,s}}$. Since $\Upsilon$ exchanges the two arcs $[\al_i^1],[\al_i^2]$ for every $i=1,\ldots,N_0$, we get that $h(x)\in \mathcal{I}(\sac{S_{h-1,r,s}})$.  
	
	Now we construct the inverse of $h$. Start with $y\in \mathcal{I}(\sac{S_{h-1,r,s}})$. Since, $y\in \sac{S_{h-1,r,s}}$, there exists a unique big simplex $\sigma_y$ such that $y\in \inte{\sigma_y}$, \ie \[  y=\sum_{j=1}^{q} s_j\, \al_j\text{, with }s_j\in (0,1), \al_j\in \sigma_y^{(0)}, \text{ for every } j=1,\ldots,q\text{ and }\sum_j s_j=1.\] Since $y\in \mathcal{I}$, it is invariant under the action of $\Upsilon$. The family of arcs in $\sigma_y^{(0)}$ project to equal or disjoint arcs in the quotient surface. Similarly, since $\sigma_y$ is big the connected components of the complement of this family of arcs are disks and they project to equal or disjoint regions. If $\al,\al'$ are two arcs in $\sigma_y^{(0)}$ that have equal weight $t$, then they project to the same arc $\be$; so $h^{-1}(t([\al]+[\al']):=t\be$.
	This concludes the proof of the lemma.
\end{proof}

\begin{corollary}\label{pacundecon}
	The pruned arc complex $\sac\gensurfn$ of a non-orientable surface $\gensurfn$ with non-decorated spikes is homeomorphic to an open ball of dimension $3h-7+3n+Q$.
\end{corollary}
\begin{proof}
	The subset $\Phi(\mathcal{I}(\ac{S_{h-1,r,s}}))$ is an open ball of dimension $3h-7+3n+Q$ — it can be parametrised by the lengths of geodesic arcs of an $\Upsilon$-invariant triangulation of $S_{h-1,r,s}$ and $\Upsilon$-invariant projective weights on the set $\Omega$. Using the isomorphism $h$, we get that \[ \sac\gensurfn=h^{-1}(\mathcal{I}(\ac{S_{h-1,r,s}}) )=h^{-1}\Phi^{-1}(\ball{3h-7+3n+Q}).\]
\end{proof}

Next, we shall prove that the pruned arc complex of a surface with decorated spikes is an open ball.

As mentioned previously, the arcs that are considered for spanning the arc complex for such a surface are either finite, separating at least one spike, or infinite with one endpoint exiting the surface through a spike. The former is referred to as an \emph{edge-to-edge} arc, while the latter is called a \emph{spike-to-edge} arc.

Firstly, we will prove the following theorem for orientable surfaces $\decogs$:
\begin{theorem}\label{pacdeco}
	The pruned arc complex $\sac \decogs$ of an orientable surface $\decogs$ with decorated spikes is an open ball of dimension $6g-7+3n+2Q$, where $Q$ is the total number of spikes.
\end{theorem}

We shall denote $S_0$ by the topological surface with genus $g$, $n$ boundary components and $q_i\geq0$ marked points on $\partial_i S_0$, $i=1,\ldots,n$ such that $\chi(S_0)<0$.
Let $\vec{\xi}=(\xi_1,\ldots, \xi_Q)$ be the set of marked points on $\partial S_0$. Let $Q(i)=\sum_{j=1}^{i}q_i$. Then we see that $S_0\smallsetminus \bigcup \{\xi_l\}_l$ is an orientable crowned hyperbolic surface. The marked points are called \emph{vertices} and the connected components of $\partial_i S \smallsetminus \{\xi_{Q(i-1)+1},\ldots, \xi_{Q(i)}\}$ are called \emph{edges}. 

Firstly, we do a topological operation on $S_0$ called the \emph{doubling} to obtain a "bigger" hyperbolic surface with boundary and without any marked points. This is done in two steps:
\begin{itemize}
	\item[Step 1:] We truncate small neighbourhoods of every marked point along embedded arcs, denoted by $V:=\{r_l\}_{l=1}^{Q}$, that join the edges adjacent to the spikes. The elements of $V$ are called \emph{$V$-edges}. Let $S$ be the resulting surface. For $i=1,\ldots,n$, when $q_i>0$, the $i$-th boundary of $S$, $\partial_i S$, is the union of $2q_i$ segments alternately partitioned into $V$-edges and the truncated boundary edges of $S_0$. When, $q_i=0$, $\partial_i S=\partial_i S_0$. The truncated boundary edges along with any closed loop in $\partial S_0$ are called $E$-edges.
	
	\item[Step 2:] Then we take a copy $S'$ of $S$ and glue it to $S$ along the $V$-edges. The final surface, denoted as $\Sigma:=S\sqcup S'/\sim$, has genus $2g$, with $2n+Q$ boundary components. If $\partial_i S_0$ had $q_i>0$ $E$-edges, then after gluing we get $q_i$ boundary components made out of two copies of every $E$-edge. 
\end{itemize}

\begin{figure}
	\frame{\includegraphics[width=7cm]{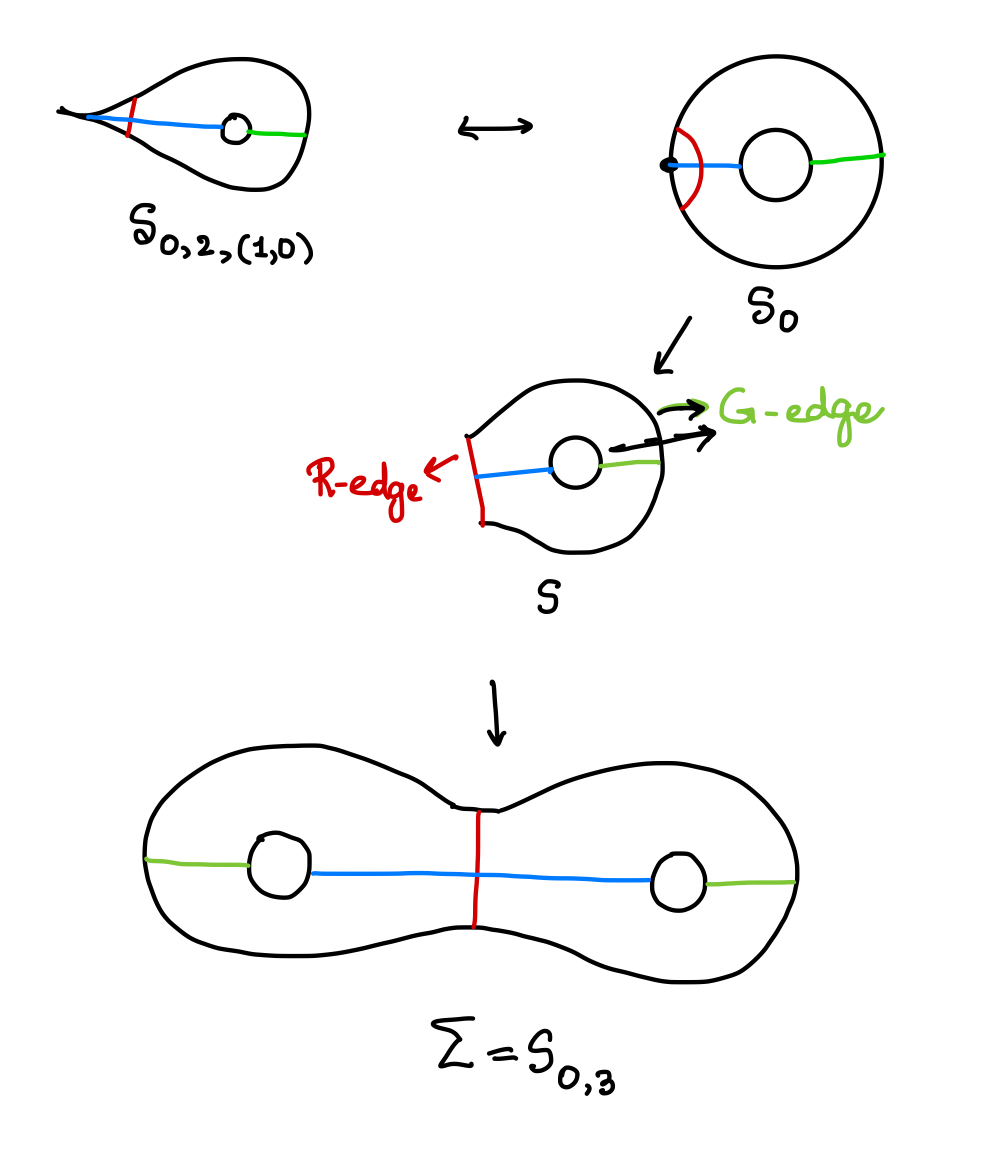}}
	\centering
	\caption{Doubling operation}
	\label{double}
\end{figure}

We get the Euler characteristic of the surface $\Sigma$, $\chi(\Sigma)=2-4g-2n<0$. So, it is hyperbolic. Since there are no spikes, we can consider all complete hyperbolic metrics with totally geodesic boundary. Its deformation space $\tei {\Sigma}$ is an open ball of dimension $12g-6+6n$. 
The surface $\Sigma$ has a degree two symmetry $\iota\in \mcg {\Sigma}$ that exchanges the two surfaces $S$ and $S'$. 

Keeping this in mind, we construct an isomorphism, denoted by $h$, between the subcomplex $\mathrm{Fix}_{\iota}(\sac \Sigma)$ of the pruned arc complex of $\Sigma$, invariant under the involution $\iota$, and the pruned arc complex $\sac {S_0}$ of $S_0$ in the following way:
At first we define it on the 0-skeleton of the arc complex and then we extend it linearly to a generic point on the pruned arc complex.
\begin{itemize}
	\item Let $e$ be an arc joining a spike $\xi$ and an edge $l$ of $\partial S_0$. Then the Step 1 above truncates $e$; in $S$ it becomes an arc, again denoted by $e$, joining the corresponding $R$-edge and the initial $G$-edge $l$. Let $e'\in S'$ be the twin arc of $e$. Finally, after the Step 2, $e:=e\sqcup e'/\sim$ becomes the arc that joins the two copies of $l$ that form the totally geodesic boundary in $\Sigma$, and transverse to the $R$-edge.  It is preserved as a set by the involution. Define $h(e):=e''$.
	\item Let $e$ be an edge-to-edge arc not in $V$. So $e$ joins two distinct boundary edges of $\partial S_0$. The Step 1 doesn't change the arc $e$. It remains disjoint from its twin $e'\subset S'$ inside $\Sigma$. Define $h(e):=\frac{e+e'}{2}$.
\end{itemize}
The map is then extended linearly over any point $x\in \sac S$.

\begin{lemma}
	The map $h:\sac {S_0}\longrightarrow \mathrm{Fix}_{\iota}(\sac \Sigma)$ is an isomorphism.
\end{lemma}
\begin{proof}
	Firstly, we show that this map is well-defined. As discussed before, a point $x\in \sac {S}$ belongs to the interior of a unique simplex $\sigma_x$ of $\ac S$. In other words, 
	\[  x=\sum_{i=1}^{p} t_i\, e_i\text{, with }t_i\in (0,1), e_i\in \sigma_x^{(0)}, \text{ for every }i=1,\ldots,p\text{ and }\sum_i t_i=1.\] 
	The union $\bigcup\limits_i e_i$ of arcs decomposes the surface $S$ into topological disks with at most one vertex. 
	
	Let $y:=h(x)=\sum_{j=1}^{q} s_j\, \al_j$. Then from the definition of $h$, it follows that $s_j\in (0,1)$, and $ \al_j\in \ack 0 \Sigma,$ for every  $j=1,\ldots,q$. 
	The family of arcs $\cur{\al_j}_j$ decomposes the surface $\Sigma$ into topological disks.  Otherwise there is a connected component $K$ in the complement in $\Sigma$ such that $\fg K\neq \{1\}$. So it is possible to find a non-trivial simple closed curve $\ga$ in $K$. Then either there was a curve in the complement of $\cur{e_i}$ in $S$ such that $\ga$ is one of its copies or the curve $\ga$ was created from a vertex-to-vertex arc by the doubling operation. None of these two cases is possible because $x\in \sac {S_0}$ and by definition of the pruned arc complex of a surface with decorated spikes, the family of arcs $\{e_i\}_i$ decomposes the initial surface $S_0$ into disks with at most one vertex. So we have that $y$ is a point of $\sac \Sigma$.
	Finally, we verify that the point $h(x)$ is $\iota$-invariant. 
	\begin{align*}
		\iota(h(e))&=\left\{
		\begin{array}{ll}
			\frac{\iota(e)+\iota(e')}{2}, & \text{ if $e$ is edge-to-edge}\\
			\iota(e'')& \text{ if $e$ is edge-to-vertex},
		\end{array}
		\right.\\
		&=\left\{
		\begin{array}{l}
			\frac{e'+e}{2}, \\
			e'',
		\end{array}
		\right.\\
		&=h(e).
	\end{align*}
	The inverse: Start with $y\in \mathrm{Fix}_{\iota}(\sac \Sigma)$. Then there exists a unique simplex $\sigma_y$ such that $y\in \inte{\sigma_y}$, \ie \[  y=\sum_{j=1}^{q} s_j\, \al_j\text{, with }s_j\in (0,1), \al_j\in \sigma_y^{(0)}, \text{ for every } j=1,\ldots,q\text{ and }\sum_j s_j=1.\] 
	Since $\iota (y)=y$, for every $j\in \{1,\ldots,q\}$, either $\iota (a_j)=a_j$ or $\iota(a_j)=\al_k$ for some $k\in \{1,\ldots, q\}\smallsetminus\{i\}$. In the former case, there exist an edge-to-vertex arc $e_j$ in $S$ and its twin $e_j'$ in $S'$ such that $a_j=e_j\sqcup e'_j/\sim$. So we define $h^{-1}(s_ja_j):=s_je_j$. In the latter case, we must also have $s_j=s_k=:s_{jk}$. Suppose that $a_j\in S$ and $a_j\in S'$. Then define $h^{-1}(s_{jk}(a_j+a_k)):=2s_{jk}a_j.$
\end{proof}

Now we shall prove Theorem \ref{pacdeco}.
\begin{proof}[Proof of Theorem \ref{pacdeco}]
	From Theorem \ref{pacundeco}, we get that there is a $\mcg \Sigma$-invariant homeomorphism:
	\begin{equation}
		\mathrm{Fix}_\iota (\sac \Sigma)\cong \mathrm{Fix}_\iota(T(\Omega)).
	\end{equation}
	The subspace $\mathrm{Fix}_\iota(T(\Omega))$ is an open ball — it can be parametrised by the lengths of the geodesic arcs of an $\iota$-invariant triangulation of $\Sigma$. Finally, using the isomorphism $h$ from above we get that $\sac {S_0}$ is an open ball of dimension $6g-7+3n+2Q$.
	
\end{proof}
Using the same method as in the previous section, we get that
\begin{theorem}\label{pacdecon}
	The pruned arc complex $\sac \decogsn$ of a non-orientable surface $\decogsn$ with decorated spikes is an open ball of dimension $3h-7+3n+2Q$.
\end{theorem}
\subsection{Tiles}\label{tilestypes}
Let $S$ be a hyperbolic surface endowed with a hyperbolic metric $m\in \tei  S$. Let $\mathcal{K}$ be the set of permitted arcs for an arc complex $\ac S$ of the surface. Given a simplex $\sigma\subset \ac S$, the \emph{edge set} is defined to be the set \[ \ed:=\set{\al_g(m)\in \al | \al\in \sigma^{(0)}},\] where $\al_g(m)$ is a geodesic representative from its isotopy class. The set of all lifts of the arcs in the edge set in the universal cover $\wt S\subset \HP$ is denoted by $\led$. The set of connected components of the surface $S$ in the complement of the arcs of the edge set is denoted by $\tile$. The lifts of the elements in $\tile$ in $\HP$ are called \emph{tiles}; their collection is denoted by $\ltile$. They are topological disks.

The sides of a tile are either contained in the boundary of the original surface or they are the arcs of $\ed$. The former case is called a \emph{boundary side} and the latter case is called an \emph{internal side}. Two tiles $d,d'$ are called \emph{neighbours} if they have a common internal side. The tiles having finitely many edges are called \emph{finite}. 

If $\sigma$ has maximal dimension in $\ac S$, then the finite tiles can be of three types:
\begin{itemize}
	\item[Type 1:]  The tile has only one internal side, \ie it has only one neighbour.
	\item [Type 2:] The tile has two internal sides, \ie two neighbours.
	\item [Type 3:] The tile has three internal sides, \ie three neighbours.
\end{itemize}

\begin{remark}
	Any tile, obtained from a triangulation using a simplex $\sigma$, must have at least one and at most three internal sides. Indeed, the only time a tile has no internal side is when the surface is an ideal triangle. Also, if a tile has four internal sides, then it must also have at least four distinct boundary sides to accommodate at least four endpoints of the arcs. The finite arc that joins one pair of non-consecutive boundary sides lies inside $\mathcal{K}$. This arc was not inside the original simplex, which implies that $\sigma$ is not maximal. Hence a tile can have at most 3 internal sides.
\end{remark}

The following are the different types of tiles possible:
\begin{itemize}
	\item When there is only one internal side of the tile, that side is an edge-to-edge arc of the original surface. The tile contains exactly one decorated spike and two boundary sides.
	\item When there are two internal sides, one of them is an edge-to-vertex and the other one is of edge-to-edge type. So the tile contains a decorated spike. 
	\item There are two possibilities in this case: either all the three internal sides are of edge-to-edge type or two of them are edge-to-vertex arcs and one edge-to-edge arcx. In the former case, the tile does not contain any vertex whereas in the latter case it contains exactly one.
\end{itemize}
\section{Infinitesimal deformations of decorated hyperbolic surfaces}
\subsection{Tile maps}
In Section 4.2 of \cite{dgk}, the authors constructed a tile-wise constant map, valued in the set of Killing fields, with respect to a given infinitesimal deformation. This map plays a central role in the proof of our main theorem. We recall the definition of this map here.

Let $\sh$ be a hyperbolic surface with holonomy representation $\rho:\fg\sh\rightarrow \pgl$ and let $u$ be an infinitesimal deformation $\rho$. Let $\sigma$ be a cellulation of $\sh$ with vertex set $\mathcal{V_\sigma}$ and edge set $\mathcal{T}_\sigma$. Then we define a \emph{tile map} $\phi_u:\ltile \rightarrow \lalg$ is a map that satisfies:

\begin{itemize}
\item $(\rho,u)$-equivariance: $\forall \ga\in \fg S, d\in \ltile,$ one has $\phi_u(\rho(\ga)\cdot d)=u(\ga)+\mathrm{Ad}(\rho(\ga))(\phi_u(d)).$
\item Consistency around internal vertices: let $d_1,\ldots, d_k $ be the tiles around one vertex. Then $$\sum_{i=1}^{k} \phi_u(d_{j+1})-\phi_u(d_j)=0.$$
\end{itemize} 

Using this vocabulary we prove the following theorem that is used in proving the openness of the admissible cone.

\begin{theorem}\label{infdef}
	Let $m=[\rho,\mathbf x, \mathbf h]$ be a decorated metric on a decorated hyperbolic surface $\sh$. Then there exists a constant $K>0$ such that for all infinitesimal deformations $u\in \tang\sh$, we have \[ \left| \mathrm{d}l_{\be}(m)(u) \right| \leq K l_{\be}(m) ||u||,\]
	where $l_{\be}$ is the length of a horoball connection $\be$. 
\end{theorem}

\begin{proof}
	Consider a lift $\tilde{\be}$ of the horoball connection $\be$ in the universal cover. It joins two decorated ideal points say $(p_1,h_1)$ and $(p_2,h_2)$. 
	\begin{figure}[!ht]
		\centering
		\frame{\includegraphics[width=8cm]{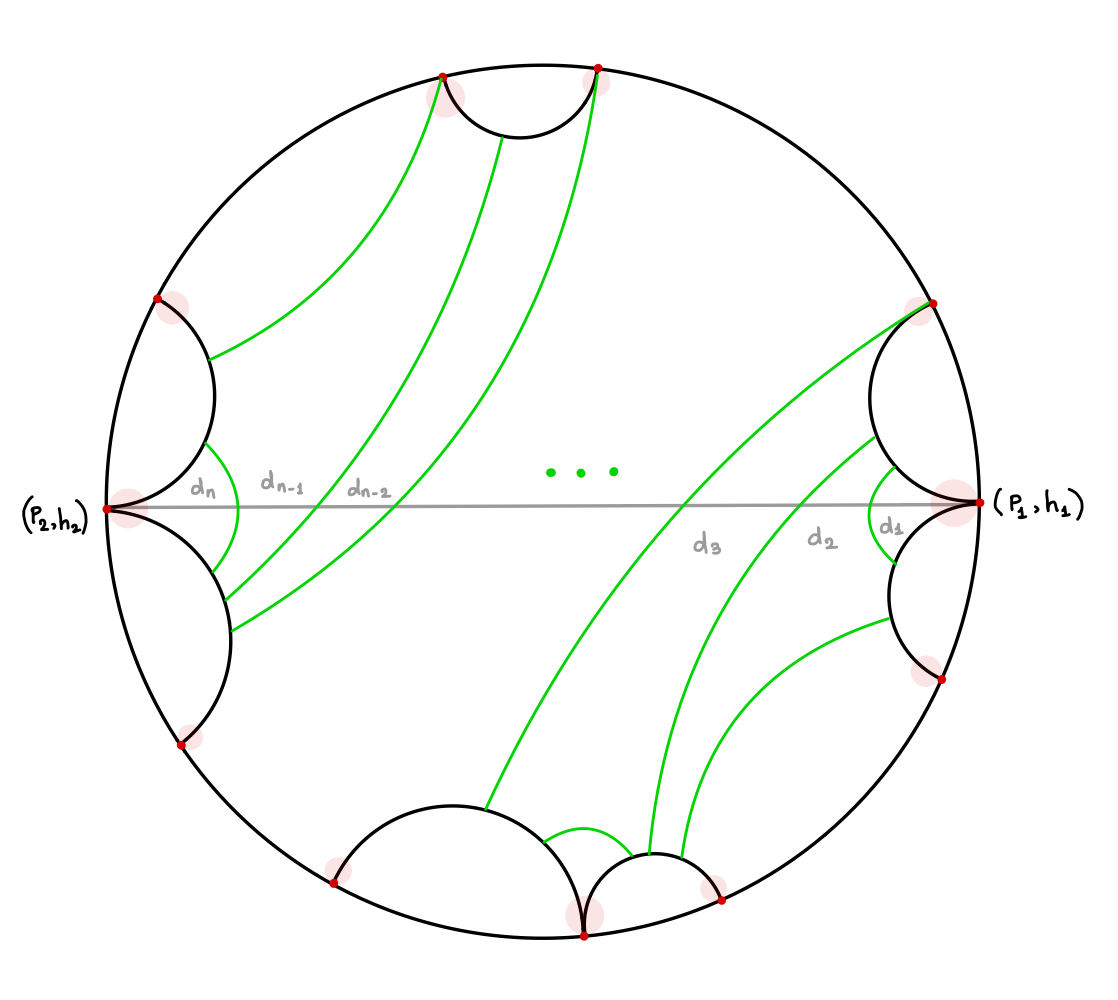}}
		\caption{Tiles containing a lift of a horoball connection}
		\label{constant}
	\end{figure}
	Now fix a triangulation $\sigma$ of $\sh$. Let $\ltile, \led$ be the associated sets of the lifts of the tiles and the edges, respectively. Let $d_1,\ldots,d_n\in \ltile$ be tiles such that $h_1\in d_1$, $h_2\in d_2$; for $j=1,\ldots,n-1$, the tile $d_j$ neighbours the tiles $d_{j+1}$ along an edge in $\led$ and contains a segment of $\tilde{\be}$. See Fig.\ref{constant}. The lifts of the edges of $\sigma$ are coloured in green. 
	
	Let $\phi_u: \ltile\longrightarrow \lalg$ be a tile map corresponding to $u$. Let $v_1, v_2 \in \Min$ be the two light-like vectors representing $(p_1,h_1)$ and $(p_2,h_2)$, respectively. Recall that the length of the horoball connection is given by $l_{\be}=\ln -\frac{\ang{v_1,v_2}}{2}$. Let $\be_t$ be the horoball connection obtained by infinitesimally deforming the metric on $\sh$ by $u$. Then the length of $\be_t$ is given by $l_{\be_t}=\ln -\frac{\ang{v^t_1,v^t_2}}{2}$, where
	\begin{align}
		v^t_1=v_1+t(\phi_u(d_1)\mcp v_1)+o(t^2),\\
		v^t_2=v_2+t(\phi_u(d_n)\mcp v_2)+o(t^2).
	\end{align}
	Let $\hat{n}$ be the unit space-like vector such that $v_2\mcp v_1=\ang{v_1,v_2} \hat{n}.$ Then we have, 
	\begin{align*}
		\ang{v^t_1,v^t_2}&=\ang{v_1,v_2}+t\{ \ang{v_1, \phi_u(d_n)\mcp v_2} +  \ang{v_2, \phi_u(d_1)\mcp v_1}\} + o(t^2)\\
		&= \ang{v_1,v_2}+t\{ \ang{\phi_u(d_n), v_2\mcp v_1} +  \ang{\phi_u(d_1), v_1\mcp v_2}\} + o(t^2)\\
		&=  \ang{v_1,v_2}+t\{ \ang{\phi_u(d_n)-\phi_u(d_1), v_2\mcp v_1} +o(t^2)\\
		&=\ang{v_1,v_2}\{1+t\ang{\sum \limits_{j=1}^{n-1}    \phi_u(d_{j+1})-\phi_u(d_j), \hat{n}}\}+ o(t^2),\\
		\Rightarrow l_{\be_t}&=l_{\be} + t\ang{\sum \limits_{j=1}^{n-1}    \phi_u(d_{j+1})-\phi_u(d_j) \hat{n}} +o(t^2). 
	\end{align*}
	
	Let $F$ be a fundamental domain for the action of $\rho(\fg{\sh}).$ Then for every $j=1,\ldots, n$, there exists $\ga_j\in\fg{\sh}$ such that $\rho(\ga_j)\cdot d_j=: \Delta_i\in F$. Then we have,
	\begin{align*}
		\ang{\phi_u(d_{j+1})-\phi_u(d_j), \hat{n}}=\ang{\phi_u(\Delta_{j+1})-\phi_u(\Delta_j), \mathrm{Ad}(\rho(\ga_j))\hat{n}}= \mathcal{O}(1),
	\end{align*}
	for every $j=1,\ldots, n$ because there are finitely many edges and vertices of the triangulation that are crossed by $\beta$. 
	
	The length of a horoball connection is the hyperbolic length of the geodesic segment subtended by the two horoballs. This segment lies in a compact part of every tile that it crosses. So the total number $n$ of the tiles $d_1,\ldots,d_n$ defined above is proportional to the length $l_\be$ of the horoball connection $\be$. 
	
	Hence from the above calculations we get that there exists a constant $K_u$ such that $l_{\be_t}\leq K_u l_{\be}$. This implies that \[\left| \mathrm{d}l_{\be}(m)(u) \right| \leq K_u l_{\be}(m),  \]
	
	Finally, let $\{ u_1,\ldots, u_N \}$ be a basis of $\tang\sh$. Let $K_{u_i}$ be the constant obtained by applying the above arguments to $u_i$. Let $K_m:=\max\{K_{u_i} \mid i=1,\ldots N\}$. Let $||\cdot||$ be any norm on $\tang\sh$. Let $u=\sum_{i=1}^{k}c_iu_i$. Then there exists $C>0$ such that $||u||>C\sum_{i=1}^{k}c_i$. Finally, we have 
	\[
	\left| \mathrm{d}l_{\be}(m)(u) \right| = \sum_{i=1}^{k} c_i \left| \mathrm{d}l_{\be}(m)(u_i) \right|\leq   l_\be(m) \sum_{i=1}^{k} c_i K_{u_i} \leq \frac{K_m}{C} ||u|| l_\be(m).
	\]
\end{proof}

Now we prove the openness of the admissible cone.
\begin{theorem}
	The subspace $\adm m$ is an open convex cone of $\tang \sh$.
\end{theorem}

\begin{proof}
	Let $K>0$ be the constant in Theorem \ref{infdef}.
	Now let $v\in \adm m$. Let $\epsilon_0:=	\inf_{\be\in \mathcal{H}}\frac{\mathrm{d}l_{\be}(m)(v)}{l_{\be}(m)}>0.$ Then for any $u\in \tang\sh$ such that $||u||<\frac{\epsilon_0}{K}$, we have that for any horoball connection $\be$
	\begin{align*}
		\frac{\mathrm{d}l_\be(m)(v+u)}{l_\be(m)}&=\frac{\mathrm{d}l_\be(m)(v)}{l_\be(m)}+\frac{\mathrm{d}l_\be(m)(u)}{l_\be(m)}\\
		&\geq \epsilon_0-K||u||\\
		&>0.
	\end{align*}
	This proves that there exists a small neighbourhood around $v$ containing only admissible deformations.
\end{proof}

\subsection{Strip deformations and their properties}\label{sd}
In this section, we recapitulate strip deformations and tile maps. We point the reader towards Section 4 in \cite{ppstrip} for the definitions of strips, strip templates and for a more detailed discussion.  

	Given an isotopy class of arcs $\al$ of a hyperbolic surface $S$ and a strip template $\{(\al_g,\pal,\wal)\}_{\al\in \mathcal{K}}$ adapted to the nature of $\al$ for every $m\in \tei {S}$, 
	we define the \emph{strip deformation} along $\al$ to be a map
	\[
	F_{\al}:\tei S \longrightarrow \tei S
	\]
	where the image $F_{\al}(m)$ of a point $m\in \tei S$ is a new metric on the surface obtained by cutting it along the $m$-geodesic arc $\al_g$ in $\al$ chosen by the strip template and gluing a strip whose waist coincides with $\pal$. If the arc is finite (resp. infinite) then the strip added is of hyperbolic (resp. parabolic) type. 
	We define the \emph{infinitesimal strip deformation}
	
	\[
	\begin{array}{cc}
		f_{\al}:&\tei {S} \longrightarrow T\tei{S}\\
		&m\mapsto [m(t)]
	\end{array}
	\]
	where the image $m(\cdot)$ is a path in $\tei {S}$ such that $m(0)=m$ and $m(t)$ is obtained from $m$ by strip deforming along $\al$ with a fixed waist $\pal$ and the width as $t\wal$.

Let $m=( [\rho, \vec{x}])\in \tei \sh$ be a point in the deformation space of the surface, where $\rho$ is the holonomy representation. Fix a strip template $\{\st \}$ with respect to $m$. Let $\sigma$ be a simplex of $\ac \sh$. Given an arc $\al$ in the edgeset $\ed$, there exist tiles $\del, \del'\in \tile$ such that every lift $\wt \al$ of $\al$ is the common internal side of two lifts $\wt \del,\wt {\del'}$ of the tiles. Also, $p_{\ga\cdot\wt\al}=\ga\cdot p_{\wt\al}$, for every $\ga\in \fg\sh$. Then the infinitesimal deformation $f_{\al}(m)$ tends to pull the two tiles $ \del$ and ${\del'}$ away from each other due to the addition of the infinitesimal strip.  Let $u$ be a infinitesimal strip deformation of $\rho$ caused by $f_{\al}(m)$. Then we have a $(\rho, u)$-equivariant tile map$\phi:\ltile \rightarrow\lalg$ such that for every $\ga\in \fg \sh$, 
\begin{align}\label{tile}
	\phi(\h\ga\cdot \wt\del)-\phi(\h\ga\cdot\wt\del')=\h\ga\cdot  v_{\wt\al},
\end{align}
where $v_{\wt\al}$ is the Killing field in $\lalg\simeq \mathscr X$ corresponding to the strip deformation $f_{\wt\al}(m)$ along a geodesic arc $\wt\al_g$, isotopic to $\wt\al$, adapted to the strip template chosen, and pointing towards $\wt\del$:
\begin{itemize}
	\item If $\al$ is a finite arc, then $v_{\wt\al}$ is defined to be the hyperbolic Killing vector field whose axis is perpendicular to $\wt\al_g$ at the point $\wt\pal$, whose velocity is $\wal$.
	\item If $\al$ is an infinite arc joining a spike and a boundary component, then $\isda$ is a parabolic strip deformation with strip template $(\al,\pal)$, and $v_{\wt\al}$ is defined to be the parabolic Killing vector field whose fixed point is the ideal point $\pal$ where the infinite end of $\wt\al$ converges.
\end{itemize}
\begin{remark}
	Such a strip deformation $f_{\al}:\tei S\longrightarrow\tang S$ does not deform the holonomy of a crowned surface if $\al$ is completely contained outside the convex core of the surface. However, it does provide infinitesimal motion to the spikes.
\end{remark}
More generally, a linear combination of strip deformations $\sum_{\al} c_\al f_{\al}(m)$ along pairwise disjoint arcs $\{\al_i\}\subset \ed$ imparts motion to the tiles of the triangulation depending on the coefficient of each term in the linear combination. A tile map corresponding to it is a $(\rho, u)$-equivariant map $\phi:\ltile \rightarrow\lalg$ such that for every pair $\del, \del ' \in \tile$ which share an edge $\al\in \ed$, the equation \ref{tile} is satisfied by $\phi$. 

\begin{definition}\label{ism}
	The \emph{infinitesimal strip map} is defined as:
	\[
	\begin{array}{ccrcl}
		\mathbb{P}f& : &	\sac S & \longrightarrow & \ptan {S}\\
		& &\sum\limits_{i=1}^{\dim \tei S} c_i \al_i&\mapsto&\bra{\sum\limits_{i=1}^{\dim \tei S}c_i f_{\al_i}(m)} 
	\end{array}
	\]
\end{definition}
where $\sac S$ is the pruned arc complex of the surface (Definition \eqref{pac}). 

Next we recall the strip width function.
Let $\sh$ be a surface with decorated spikes with a metric $m$ and a strip template $\st$. Let  a strip deformation $\isda$ along a finite arc $\al$, with strip template $\st$. Let $w_\al(p)$ be the width of the strip at the point $p\in \al_g$. Suppose that $v_{\wt\al}$ is the Killing field acting across $\wt{\al_g}$ due to the strip deformation. Then $w_\al(p)=\norm{v_{\wt\al}\mcp p}.$

\begin{definition}
	Let $\sh$ be a surface with decorated spikes with a metric $m$ and corresponding strip template $\{\st \}$. Let $x=\sum_{i=1}^{N_0} c_i\al_i$ be a point in the pruned arc complex $\sac \sh$. Then the \emph{strip width function} is defined as:
	\[ 
	\begin{array}{rrcl}
		w_x:&\supp x&\longrightarrow&\R_{>0}\\
		& p&\mapsto&c_iw_{\al_i}(p),
	\end{array}
	\]
\end{definition}

The following theorem was proved in the case of decorated polygons in \cite{ppstrip}.
\begin{theorem}\label{lenderiv}
	Let $\sh$ be a surface with decorated spikes endowed with a decorated metric $m$ and a corresponding strip template $\st$. Let $x\in \sac \sh$ and $\ga$ be a closed curve or a horoball connection of $\sh$ intersecting $\supp x$. Then, 
	\begin{equation}\label{lengthderiv}
		\mathrm{d}l_{\ga}(f(x))=\sum\limits_{p\in \ga \cap \supp x} w_x(p) \sin \angle_p( \al_g, \supp x) > 0.
	\end{equation}  
\end{theorem}

\begin{lemma}\label{inject}
	Let $x$ be a point of a hyperbolic surface $S$. Let $B$ be a geodesic ball centered at $x$ with radius $r$, where $r$ is the injectivity radius of the surface. Then for every pair of distinct lifts $B_1,B_2$ of $B$ in the universal cover of $S$, we have that $B_1\cap B_2=\emptyset$.
\end{lemma}
\begin{proof}
	Let $x,B,B_1,B_2$ be as in the hypothesis. 
	Then, $B=\exp_x(B(0,r))$. For $i=1,2$, let $x_i\in \HP$ be the center of $B_i$. Then $x_2=\ga\cdot x_1$ for some $\ga\in \fg S$.
	
	If possible, let $\wt z\in B_1\cap B_2$. Since the action of $\fg S$ is free, the point $z':=\ga\cdot\wt z$ lies inside $B_2$. Join $x_2$ with $z$ and $z'$ by geodesic segments $l_1,l_2$, respectively. On the surface $S$, the path $l_1\cup l_2$ is mapped to a loop based at $z$. Since $B$ is the embedding of a ball by the exponential map, this loop is trivial. So $l_1\cup l_2$ is a loop in $\wt S$ based at a lift $\wt z$, which is a contradiction.
\end{proof}

\begin{lemma}\label{ineq}
	Let $\sh$ be a hyperbolic surface with decorated spikes and a metric $m\in \tei \sh$.
	Then there exists $M>0$ such that for every non-trivial closed geodesic $\ga$ and for every non-trivial geodesic arc $\al$, the following inequality holds:
	\begin{align}
		\sum_{p\in \ga\cap\al} w_\al (p)\leq M l_{\ga}(m).
	\end{align}
\end{lemma}
\begin{figure}[h!]
	\includegraphics[width=7cm]{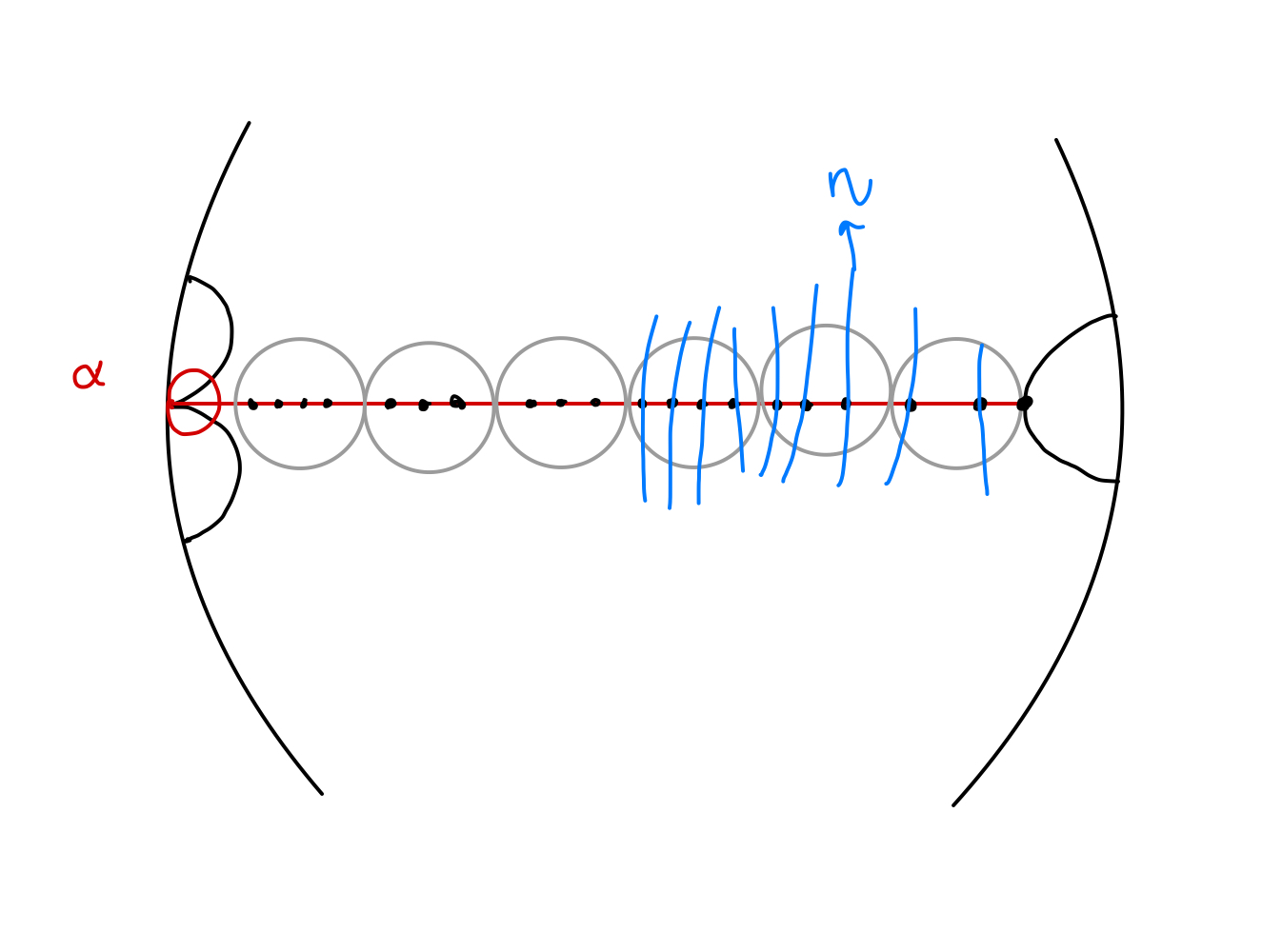}
	\centering
	\label{impact}
	\caption{Intersection of the arc and the curve}
\end{figure}

\begin{proof}
	Let $\ga$ and $\al$ be as in the hypothesis. We further suppose that $\al$ is an infinite arc joining a decorated spike and a boundary component of $S$.
	We prove that there exists a positive constant $M_0$ such that for every unit segment $\eta$ and for every arc $\al$, the following inequality holds.
	\begin{align}
		\sum_{p\in \eta\cap\al} w_\al (p)\leq M_0.
	\end{align}
	Given such a unit segment $\eta$, let $\eta\cap \al=\{p_1,\ldots,p_k\}$. We order them so that $p_i$ lies closer to the horoball than $p_j$ if and only if $i>j$.
	Take a lift $\wt\al$ of $\al$ in the universal cover inside $\HP$. Since $\al$ is embedded, for every $i=1,\ldots,k$, there is exactly one lift $\wt {p_i}$ of $p_i$ that lies on $\wt\al$. Let $r_0$ be the injectivity radius of the surface. Cover the entire arc $\wt\al$ with balls $\{B_j\}_{j\in J}$ of radius $r:=\frac{r_0}{2}$, such that two consecutive balls are tangent to each other. 
	
	We claim that each ball contains at most $M_0$ number of intersection points, where $M_0:=[\frac{1}{r_0}]+1$. Consider one lift of $\eta$ and a ball $B_r$ in the above covering. Then a reformulation of the claim is that there are at most $M_0$-many balls that are in the same orbit as $B$. Now each of these balls are contained in the bigger ball $B_{2r}$ with the same center and of radius $r_0$ and by Lemma \ref{inject}, no two of them can intersect. Thus the maximum number of balls $B_r$ in the same orbit intersecting $\eta$ is $M_0$.
	
	Next, we know that $w_{\wt\al} (\wt p_i)=\e^{L_i}$, where $L_i$ is the negative arc coordinate of $\wt {p_i}$ along $\wt\al$. Then, $w_{\wt\al} (\wt {p_i})$ decreases exponentially as $i$ increases. In every ball, the maximum value of $w_{\wt\al}$ is attained at the rightmost point. Two such points in two consecutive balls are at most $r_0$ distance apart because the balls are of radius $\frac{r_0}{2}$ and tangent to each other. Inside the first ball, the maximum value of $w_{\wt\al}$ can be at most $1$, if the point of unit impact is an intersection point. 
	
	So we have,
	\begin{align}
		\sum_{p\in \eta\cap\al} w_\al (p)=\sum_{i=1}^k w_{\wt\al} (\wt p_i)=\sum_{i=1}^k \e^{L_i}\leq M_0(1+ \e^{-r_0}+\ \e^{-2r_0}+\ldots)=\frac{M_0}{1-\e^{-r_0}}.
	\end{align}
	
	Finally, taking the sum over all the unit segments of $\ga$, we get that 
	\begin{align*}
		\sum_{p\in \ga\cap\al} w_\al (p)\leq M l_{\ga}(m),
	\end{align*}
	where $M:=\frac{M_0}{1-\e^{-r_0}}$.
\end{proof}

\begin{lemma} \label{maxangle}Let $S_c$ be a compact hyperbolic surface equipped with a metric $m\in \tei {S_c}$.
	Then for every $\epsilon>0$ there exists $M>0$ such that whenever a geodesic arc $\al$ has $m$-length $l_\al(m)>M$, there exists a closed geodesic $\ga$ on the surface such that it intersects $\al$ as well as every geodesic arc, that is disjoint from $\al$, at angle less than $\epsilon$.
	
\end{lemma}
\begin{proof}
	Given $\epsilon>0$, there exists $N\in \N$ such that $N\mathrm{diam}(S_c)\epsilon>3\,\mathrm{area}(S_c)$. Take $M=N\mathrm{diam}(S_c)$. Consider an arc $\al$ of length $M$ and its $\epsilon$-neighbourhood $V_\epsilon (\al)$. The area of $V_\epsilon (\al)$ is at least $2M\epsilon$. So $V_\epsilon (\al)$ cannot be embedded inside the surface — it self-overlaps threefold. It follows that there exists a segment $\eta$ of the arc $\al$ such that its length is $N'\mathrm{diam}(S_c)$ for $1<N'<N$ and its endpoints lie at a distance $2\epsilon$ from each other, with the velocities at those points being parallel. Join the two endpoints by a geodesic segment to get a closed loop. Then choose the unique closed geodesic $\ga$ in its homotopy class. We claim that $\ga$ satisfies the condition of the lemma. Firstly, we show that 
	\begin{equation}\label{max}
		\max\limits_{p\in \ga \cap \al}\angle_p(\ga,\al)<\epsilon.
	\end{equation}
	Let $\wt\ga$ be a infinite geodesic lift of $\ga$. Let $\wt\eta$ be a lift of the arc segment $\eta$ and consider its $\rho(\ga)$-orbit. For every $i\in \Z$, the two endpoints of $\rho(\ga)^i\cdot\wt\eta$ are $\epsilon$-close to one endpoint of $\rho(\ga)^{i-1}\cdot\wt\eta$ and one endpoint of $\rho(\ga)^{i+1}\cdot\wt\eta$. Let $p_i:=\wt\ga\cap\rho(\ga^i)\cdot\wt\eta$. The entire geodesic $\wt\ga$ is contained in the union $V:=\bigcup_{i\in\Z}V_\epsilon(\rho(\ga^i)\cdot \wt\eta)$ of the $\epsilon$-neighbourhoods of the lifts of $\eta$. So the angle of intersection at $p_i$ satisfies $$\angle_{p_i}(\wt\ga, \rho(\ga^i)\cdot\wt\eta)<\epsilon.$$
	
	Now let $\al'$ be an arc disjoint from $\al$ that intersects $\ga$. Then the point of intersection, denoted by $p$, lies inside $V_\epsilon (\al)$. Then from eq.\ \eqref{max} we have that $\al'$ intersects $\ga$ at an angle less than $\epsilon$. 
\end{proof}
The following lemma is an analogue of Proposition 2.3 in \cite{dgk}.
\begin{lemma}\label{minangle}
	Let $\sh$ be a hyperbolic surface with decorated spikes endowed with metric $m\in \tei \sh$. For any choice of minimally intersecting geodesic representatives $\{\al\}$ whose finite endpoints lie outside the horoball decoration of the spikes, there exists $\theta_0 \in (0,\frac{\pi}{2}]$ such that all the arcs intersect the boundary of the surface at an angle greater or equal to $\theta_0$.
\end{lemma}

\section{Main theorem}\label{main}
The goal of this section is to prove our parametrisation theorem:
\begin{theorem} \label{maindec}
	Let $\sh=\decogs$ or $\decogsn$ be a hyperbolic surface with decorated spikes. Let $m\in \tei{\sh} $ be a decorated metric. Fix a choice of strip template $\{\st \}_{\al \in\mathcal K}$ with respect to $m$. Then, the infinitesimal strip map $\mathbb{P}f:\sac {\sh}\longrightarrow \ptan {\sh}$ is a homeomorphism on its image  $\mathbb{P}^+(\adm m)$.
\end{theorem}
Firstly, we show that the map $\psm$ is a local homeomorphism and then we show that it is proper. This proves that the map is a covering map between two open balls of the same dimension, which implies that it is a homeomorphism.
\subsection{Codimension 0}
In this section, we show thatnthe map $\mathbb{P}f$ is a local homeomorphism around points $x\in \sac{S}$ such that $\codim{\sigma_x}=0$. For that, we first define longitudinal motions.

Let $\mathbf a, \mathbf b\in \Min$ be two linearly independent future-pointing light-like vectors whose projective images are denoted by $A,B$. Let $AB$ be the unique hyperbolic geodesic joining these two points. 
\begin{definition}
	Given a Killing vector field $X\in \lalg$, the \emph{longitudinal motion} $X_l$ imparted by $X$ to the geodesic $AB$ is defined as
	\begin{equation}
		X_l:= \bil{\mathbf {v}_X}{\frac{\mathbf a\mcp \mathbf b}{\norm{(\mathbf a\mcp \mathbf b)}}},
	\end{equation}
	where $\mathbf v_X$ is the vector in $\Min$ corresponding to $X$.
\end{definition}
The motion is called "longitudinal" because the vector $X_l$ is equal to the component of $X(p)$ along the direction of the line AB, for every point $p \in \HP$ lying on AB.

\begin{figure}
	\centering
	\frame{\includegraphics[width=10cm]{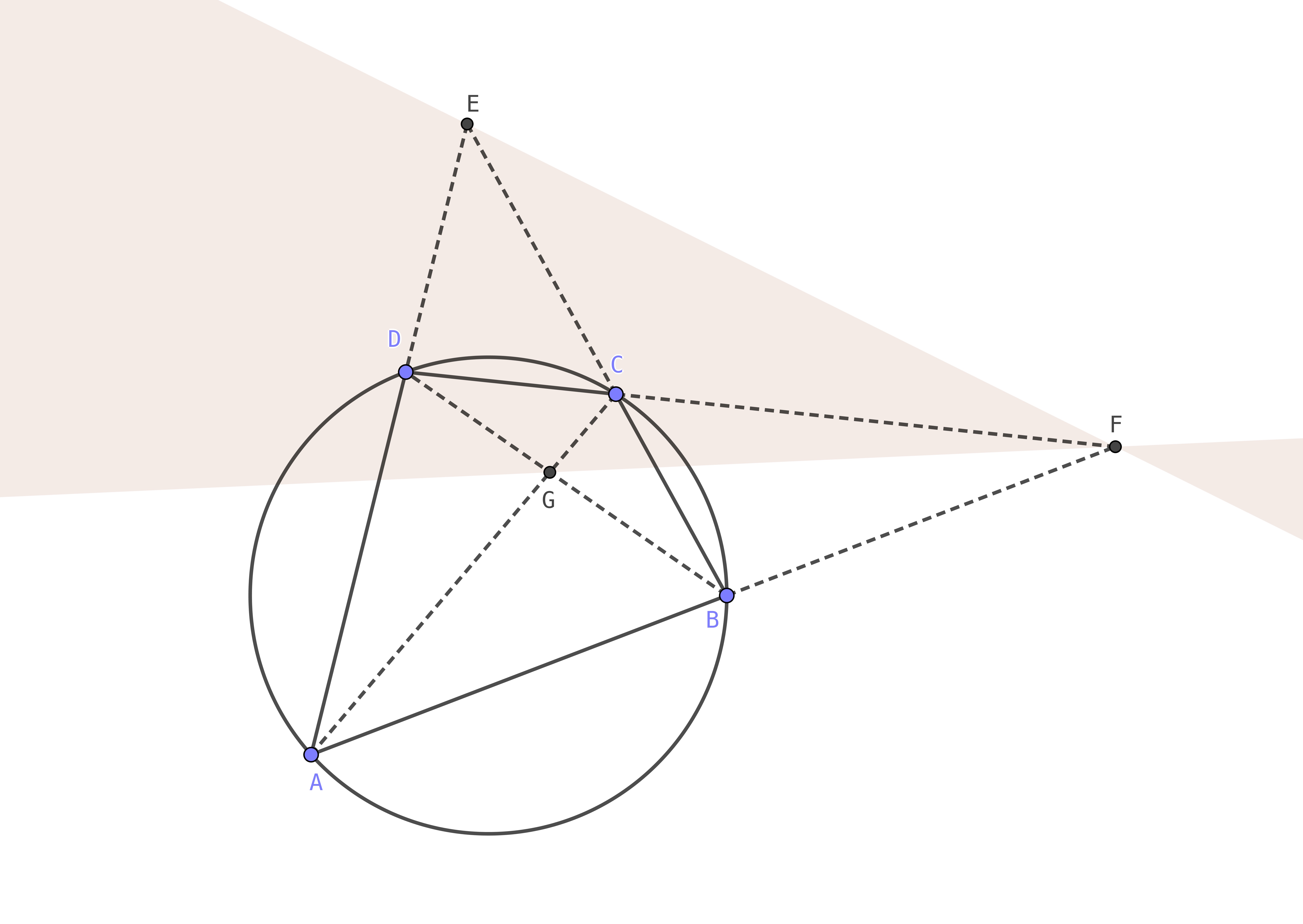}}
	\caption{The shaded region is the bigon of Lemma \ref{longi}}
	\label{longifig}
\end{figure}

\begin{lemma}\label{longi}
	Let $A,B,C,D$ be four ideal points ordered in anti-clockwise manner. Let $AB$ and $CD$ be two disjoint geodesics in $\HP$. Let $E,F,G$ be the intersection points $\oarr {AD} \cap \oarr {BC}$, $\oarr {AB}\cap\oarr {CD}$ and $\oarr{ AC}\cap\oarr {BD}$, respectively. Then the set of all Killing fields that impart at least the same amount (absolute value) of longitudinal motion to $AB$ as to $CD$, is given by the bigon bounded by $\oarr {GF}$ and $\oarr {EF}$, that contains the segment $CD$.
\end{lemma}
\begin{proof}
	Firstly, we prove that the Killing vector fields who projective images are one of $E,F,G$, impart equal longitudinal motion to both $AB$ and $CD$. 
	Let $\mathbf a,\mathbf b,\mathbf c,\mathbf d$ be future-pointing light-like vectors in the preimages of $A,B,C,D$ such that
	\[
	\begin{array}{ll}
		\mathbf	a=(-\cos{\theta},-\sin{\theta}),1), &\mathbf b=(\cos\theta,-\sin \theta,1)\\
		\mathbf	c=(\cos{\theta},\sin{\theta},1), &\mathbf d=(-\cos{\theta},\sin{\theta},1),
	\end{array} 
	\] where $\theta\in [0, \frac{\pi}{2}]$. The existence of $theta$ can be assumed up to applying a hyperbolic isometry to the quadruple $(A,B,C,D)$ because any cross-ratio is realized by some rectangle.
	Then the points $A,B,C,D$ form a rectangle in the projective plane. Also, we have that
	\begin{eqnarray}\label{e}
		E=\bra{(\mathbf a\mcp\mathbf d)\mcp (\mathbf c\mcp\mathbf  b)},& F=\bra{(\mathbf a\mcp \mathbf b)\mcp (\mathbf c\mcp \mathbf d)} ,
	\end{eqnarray}
	\begin{eqnarray}\label{g}
		G=\bra{(\mathbf a\mcp\mathbf c)\mcp (\mathbf b\mcp\mathbf d)} .
	\end{eqnarray}
	It follows directly from the definition of cross product that any Killing vector field that is a preimage of $F$ imparts no longitudinal motion either to $AB$ or to $CD$. 
	
	By inserting the coordinates of $\mathbf a,\mathbf b,\mathbf c,\mathbf d$ in the formulae \eqref{e},\eqref{g}, we get that 
	\[
	\begin{array}{ll}
		\mathbf a\mcp \mathbf c=2(-\sin \theta, \cos\theta, 0),&\mathbf b\mcp \mathbf d=2(-\sin \theta, -\cos\theta, 0),\\
		G=[(0,0,-4\sin(2\theta))].
	\end{array}
	\]
	Furthermore, we have that 
	\[
	\begin{array}{c}
		\mathbf a\mcp \mathbf b=(0,2\cos\theta,-\sin(2\theta)),\\
		\mathbf c\mcp\mathbf d=(0,-2\cos\theta,-\sin(2\theta)),\\
		{\norm{\mathbf a\mcp\mathbf  b}}^2={\norm{\mathbf c\mcp \mathbf d}}^2=(1+\cos(2\theta))^2.
	\end{array}
	\]
	If we take any Killing vector field $X_G=k((\mathbf a\mcp \mathbf c)\mcp (\mathbf b\mcp \mathbf d))$ in the preimage of $G$, $k\in\R^*$, we get that \[\bil{X_G}{\frac{\mathbf a\mcp \mathbf b}{\norm {\mathbf a\mcp\mathbf  b}}}=\bil{X_G}{\frac{\mathbf c\mcp\mathbf d}{\norm {\mathbf c\mcp\mathbf  d}}}=\frac{-4\sin^2(2\theta)}{(1+\cos(2\theta))}=-8\sin^2\theta.\]
	Finally, we calculate the coordinates of $E$:
	\[
	\begin{array}{rl}
		\mathbf a\mcp \mathbf d=r^2(-2\sin\theta, 0, \sin(2\theta)),&\mathbf c\mcp\mathbf  b=r^2(2\sin\theta, 0, \sin(2\theta))\\
		E=[4r^4(0,\sin\theta\sin(2\theta),0)].
	\end{array}
	\]
	
	If $X_E=k'(a\mcp d)\mcp (c\mcp b)$ for some $k'\in \R$, then we have \[\bil{X_E}{\frac{\mathbf a\mcp \mathbf b}{\norm {\mathbf a\mcp\mathbf  b}}}=-\bil{X_E}{\frac{\mathbf c\mcp\mathbf d}{\norm {\mathbf c\mcp\mathbf  d}}}=4\cos^2\theta.\]
	
	By linearity, we get that the Killing vector fields whose projections lie on the straight line $\oarr{EF}$ and $\oarr{GF}$ impart equal or opposite longitudinal motions to $AB$ and $CD$. Now, any Killing field, whose projective image lies on the straight line $c\mcp d$, imparts zero motion to $CD$ and non-zero motion to $AB$. Since the longitudinal motion on a given line is a linear function of the Killing field, the bigon containing the segment $CD$ is the desired one.
\end{proof}
\begin{theorem}\label{codim0deco}
	Given a triangulation $\sigma$ of a hyperbolic surface with decorated spikes $\sh=\decogs$ or $\decogsn$ with corresponding edge set $\ed$, the set of infinitesimal strip deformations $B=\{   \isd| e\in \ed  \}$ forms a basis of $\tang \sh$.
\end{theorem}
\begin{proof}
	We start with a neutral tile map $\phi_0:\ltile\longrightarrow \lalg$ for the triangulation $\sigma$, representing the linear combination \[\sum_{e\in \ed}c_e \isd=0,\] and we show that the maximal longitudinal motion along any arc of the triangulation is zero. 
	
	Let $e$ be a common internal edge of two tiles $d,d'\in \ltile$. From the definition of tile maps we know that when $e$ is spike-to-edge, the difference $\np d-\np{d'}$ is a light-like point in the plane $L_e$, and when $e$ is an edge-to edge arc, the difference is a space-like point in $L_e$. We claim that the longitudinal motions imparted to $e$ by $\np d$ and $\np {d'}$ are equal.
	
	Indeed,
	we can decompose the Minkowski space as $$\Min=L_e\oplus \du{L_e},$$ where $L_e$ is the plane $\mathbb{P}^{-1}(\oarr e)$ and $L_e^\perp$ is the $\bil{\cdot}{\cdot}$-dual of $L_e$. Then $\np d=\mathbf v_t+\mathbf v_l$ and $\np {d'}=\mathbf {v}_t'+\mathbf {v}_l'$ with $\mathbf v_t,\mathbf {v}_t' \in L_e$ and $\mathbf v_l,\mathbf {v}_l'\in \du{L_e}$.
	Now from the definition of tile maps we have that the vector $\phi(d)-\phi(d')$ is a space-like point of $L_e$.  Hence, $\mathbf v_l=\mathbf v_l'$, proving our claim.

	Moreover, when $e$ is spike-to-edge, the Killing fields $\np d,\np{d'}$ are parabolic preserving the spike as well as the horoball decoration. So the longitudinal motion along $e$ is zero in this case. It remains to show that the maximal longitudinal motion along any edge-to-edge arc is zero. 
	\begin{figure}
		\centering
		\frame{\includegraphics[width=10cm]{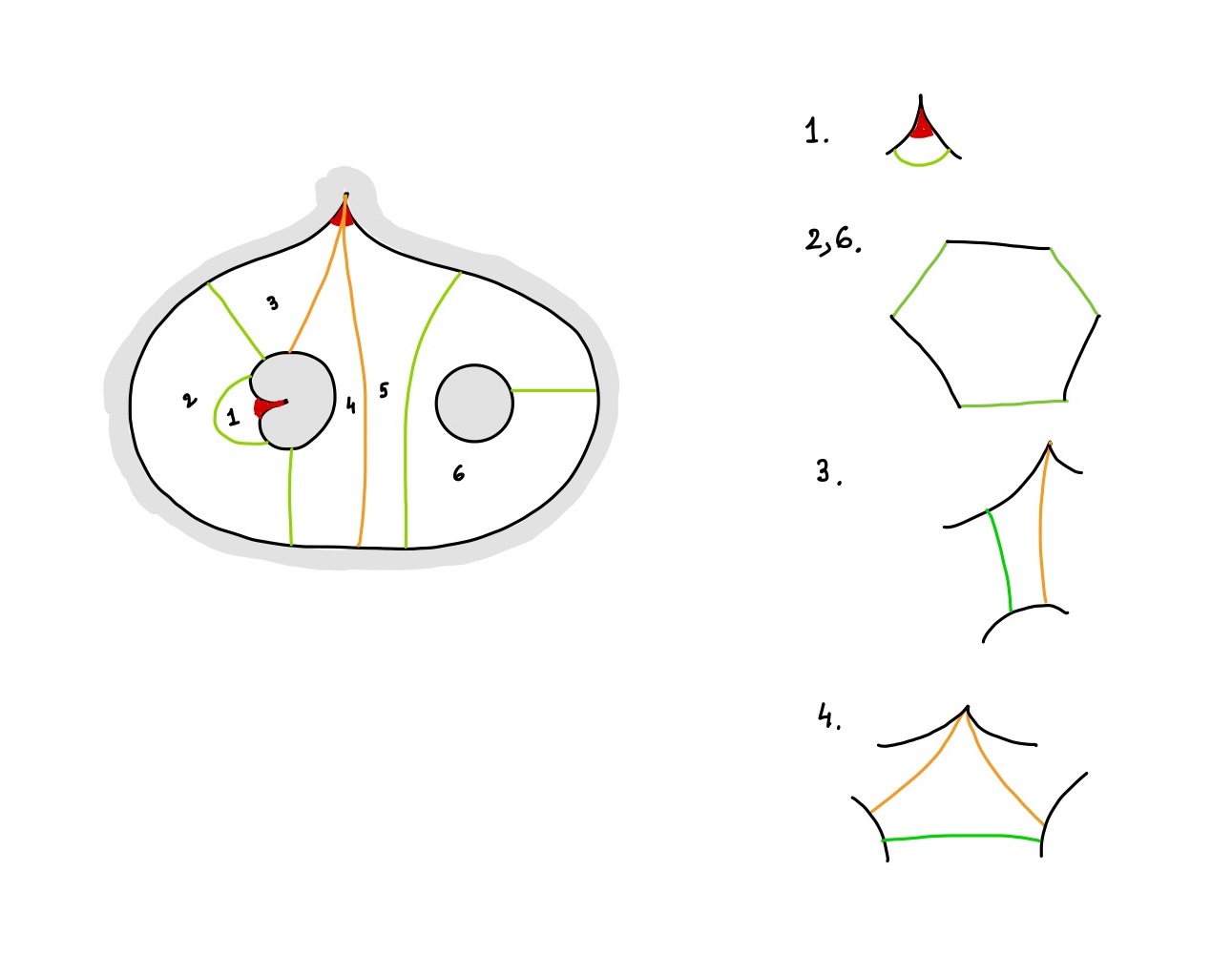}}
		\caption{Tiles of a triangulation of $S_{0,3}^{(1,1,0)}$}
		\label{hptiles}
	\end{figure}
	We use the following theorem:
	\begin{lemma}\label{maxlongideco}
		Suppose that $e$ is an edge-to-edge arc with maximal longitudinal motion. Let $d\in \ltile$ be a tile with $e$ as an internal edge. Then, the point $[\np d]$ is contained in the interior of the projective triangle based at $e$, containing $d$.
	\end{lemma}
	\begin{proof}
		Figure \ref{hptiles} shows the different tiles formed after the triangulation of the surface.
		\begin{itemize}
			\item Suppose that $d$ is of type one, \ie it is a triangle with one decorated ideal vertex and one internal edge which is edge-to-edge. (Topmost tile in Fig.\ \ref{hptiles}). The point $[\np d]$ is given by the ideal vertex, which lies inside the desired triangle.
			\item Suppose that $d$ is of type two, \ie it is a quadrilateral with one decorated ideal vertex and two internal edges, one of which is edge-to-edge (Third tile from the top in Fig.\ \ref{hptiles}). Once again, the point $[\np d]$ is given by the ideal vertex, which lies inside the desired triangle.
			\item Suppose that $d$ is of type three. Firstly, we suppose that it is a pentagon with one decorated ideal vertex and three internal edges, one of which is edge-to-edge (Fourth tile in Fig.\ \ref{hptiles}). Once again, the point $[\np d]$ is given by the ideal vertex, which lies inside the desired triangle. Finally, if $d$ is a hexagon with three edge-to-edge arcs (second tile from the top in Fig.\ \ref{hptiles}), the proof is identical to the proof of Claim 3.2(0) in \cite{dgk}.
		\end{itemize}
		This finishes the proof of Lemma \ref{maxlongideco}.
	\end{proof}
	Now let $e$ be an internal edge of two neighbouring tiles $d,d'\in \ltile$, such that $e$ has maximal (non-zero) longitudinal motion. So $e$ is an edge-to-edge arc. By Lemma \ref{maxlongideco}, the points $[\np d]$ and $[\np {d'}]$ belong to two projective triangles whose interiors are disjoint. If $c_e\neq 0$, then $[\np d- \np{d'}]$ must be a point in $\oarr e\backslash \cHP$. But any line joining $[\np d]$ and $[\np {d'}]$ intersects $\oarr e$ inside $\HP$. So we arrive at a contradiction. Hence, the longitudinal motion along every edge-to-edge arc is zero.
	
	Now we prove that $\np d=0$ for every $d\in \ltile$. Every tile $d$ of the triangulation has an internal edge-to-edge arc. Suppose that $d$ has a decorated ideal vertex $p$. Then, either $\np d=0$ or $\np d\in \pinv p$. Let $e$ be an internal edge-to-edge arc, with endpoints $A,B\in \HPb$. Let $\mathbf a\in \pinv A$ and $\mathbf b\in \pinv B$ be future pointing light-like vectors. Since the longitudinal motion along $e$ is zero we have that $$\bil{\np d}{\frac{\mathbf a \mcp \mathbf b}{\norm {\mathbf a \mcp \mathbf b}}}=0,$$ which is possible only if $\np d=0$.   
	Finally suppose that $d$ is a hexagon. Then it has three internal edges, denoted by  $e_1,e_2,e_3$. Choose space-like vectors $\mathbf v_i\in \pinv {\du {e_i}}$ for $i=1,2,3$. Then from above we know that the longitudinal motions along its three internal (pairwise disjoint in $\HP$) edges are zero. So $\np d=0$.
	
\end{proof}
\subsection{Codimension 1,2}
The proofs for the local homeomorphism of $\mathbb P f$ around points belonging to the interiors of simplices of codimension more than 0 are identical to that in the cases of decorated (once-punctured) polygons, Sections 5.2-5.3  in \cite{ppstrip}.

\subsection{Properness}

In this section we prove that the projectivised strip map $\mathbb{P}f$ is proper.
\begin{theorem}\label{properdeco}
	Let $\sh$ be a hyperbolic surface with decorated spikes. Let $m\in \tei \sh$. Then
	the projectivised strip map $\mathbb{P}f:\sac \sh \longrightarrow \mathbb{P}^+(\adm m)$ is proper.
\end{theorem}
\begin{proof}
	Let $(x_n)_n$ be a sequence in the pruned arc complex $\sac \sh$ such that $x_n\rightarrow \infty$: for every compact $K$ in $\sac \sh$, there exists an integer $n_0\in \N$ such that for all $n\geq n_0$, $x_n\notin K$.  We want to show that $\mathbb{P}f(x_n) \to \infty$ in the projectivised admissible cone $\mathbb{P}^+\adm m$. Recall that the admissible cone $\adm m$ is an open convex subset of $\tang S$. Its boundary $\partial \adm m$ consists of $\vec{0}\in \tang S$ and is supported by hyperplanes (and their limits) given by the kernels of linear functionals $\mathrm{d}l_{\be}:\tang S \longrightarrow \R$, where $\be$ is a horoball connection or a non-trivial closed geodesic of the surface. It suffices to show that $f(x_n)$ tends to infinity (in the sense of leaving every compact subset) inside $\adm m$ but stays bounded away from $\vec{0}$ so that $\mathbb{P}f(x_n)$ tends to infinity in $\mathbb{P}^+\adm m$. 
	
	From Lemma \ref{ineq}, we get that there exists a constant $M'>0$ depending on the normalisation such that for every closed geodesic $\ga$ and every point $x\in \ac \sh$ the following inequality holds
	\begin{align}\label{ineqnorm}
		\sum_{p\in \ga\cap \supp {x}} w_{x}(p)\leq M' l_{\ga}(m),
	\end{align} where $w_x: \supp x\to \R_{>0}$ is the strip width function. Let $K(S_c)$ be a compact neighbourhood of the convex core $S_c$ of the surface. Then every arc has bounded length outside $K(S_c)$: there exists $C>0$ such that for every geodesic arc $\al$, $l_{\al\smallsetminus K(S_c)}<C$.
	Given $\epsilon>0$, we get a constant $M>0$ from Lemma \ref{maxangle} applied to $\frac{\epsilon}{M'}$. Define $$\mathcal{K}_M:=\{\al \in \mathcal{K}\mid l_{\al}(m)\leq M+C\}.$$ Since there exist only finitely many geodesic arcs in $\sh$ (and hence finitely many permitted arcs) up to any given length, we have that $\mathcal{K}_M$ is finite.  Consequently, the set $\Sigma_M$ of simplices in $\ac \sh$ spanned by the arcs in $\mathcal{K}_M$ is also finite. We will show that there exists $n_1\in \N$ such that for every $n\geq n_1$, there exists a closed geodesic $\ga(n)$ that satisfies:
	\begin{equation}\label{curvedeform}
		\frac{\mathrm{d}l_{\ga(n)}(f(x_{n}))}{l_{\ga(n)}(m)}<\epsilon.
	\end{equation} 
	It is enough to prove the above inequality for two types of subsequences of $(x_n)_n$— a subsequence whose terms live in one of the finitely many simplices in $\Sigma_M$ and a subsequence whose every term lies in simplices outside $\Sigma_M$. Finally, in both the cases we show that $f(x_n)$ does not converge to $\vec 0$.
	\begin{enumerate}
		\item[Case 1:] Consider a subsequence $(y_n)_n$ of $(x_n)_n$ such that $y_n\in \sigma$ spanned by the arcs $\{\al_1,\ldots,\al_{N}\}\subset \Sigma_M$, where $N\leq\dim \tei \sh$. Since $y_n\to\infty$, it has a subsequence that converges to a point $y\in \ac {\sh}\smallsetminus\sac\sh$. So $y_n$ is of the form: \[ y_n=\sum_{i=1}^{N}t_i(n)[\al_{i}], \text {with } t_i(n)\in (0,1] \text{ and }\sum_{i=1}^{N}t_i(n)=1, \] 
		and the limit point $y$ is then given by: 
		
		\[ y=\sum_{i=1}^N t_{i}^\infty[\al_{i}],\]  where there exists $\mathcal{I}\subsetneq \{1,\ldots,N\}$ such that \[\text{ for } i\in \mathcal I,\,t_{i}(n)\mapsto t_{i}^\infty\in (0,1],\text{ and }\sum_{i\in \mathcal{I}}t_{i}^\infty=1,\] 
		\[\text{ for } i\in \{1,\ldots,N\}\smallsetminus \mathcal{I},\, t_{i}(n)\to t_i^\infty=0.\] 
		Since $y\in \ac {\sh}\smallsetminus\sac\sh$, in the complement of $\supp y=\bigcup_{i\in\mathcal{I}}\al_i$, there is either a loop or a horoball connection, denoted by $\be$. By construction, $\be$ intersects only the arcs $\{\al_i\}_{i \notin \mathcal{I}}$. By continuity of the infinitesimal strip map $f$ on $\sigma$, the sequence $(f(y_n))_n$ converges to $f(y)\in \partial\adm m$ and $$\mathrm{d}l_{\be}(f(y))=\sum_{i \notin \mathcal{I}} t_i^\infty \mathrm{d}l_{\be}(f_{\al_i}(m))=0.$$ Hence $f(y)$ fails to lengthen $\be$. 
		
		Next we show that $f(y)\neq 0$. Let $\ga$ be the boundary component containing one endpoint of an arc $\al_i$ for $i\in \mathcal I$. Then we have
		\begin{align*}
			\mathrm{d}l_{\ga}(f(y))&=\sum\limits_{p\in \ga \cap \supp {y}} w_{y}(p) \sin \angle_p( \ga, \supp {y})\\
			&\geq t_i^\infty w_{\al_i}(p)\sin \angle_p( \ga, \supp {y})\\
			&>0.
		\end{align*}
		\item[Case 2:] Consider a subsequence $(z_n)_n$ such that for every $n\in \N$ there exists an arc $\al_n\subset \supp{z_n}$ with $l_{\al_n}(m)>M+C$. So $l_{\al\cap K(S_c)}>M$. 
		From Lemma \ref{maxangle}, there exists a geodesic, denoted by $\ga(n)$, which satisfies
		\begin{equation}\label{mangle}
			\theta_0:=\max\limits_{p\in \ga(n) \cap \supp {z_n}} \angle_p (\supp {z_n}, \ga(n))<\frac{\epsilon}{M'}.
		\end{equation}
		Thus we have
		\begin{align*}
			\mathrm{d}l_{\ga(n)}(f(z_n))&=\sum\limits_{p\in \ga(n) \cap \supp {z_n}} w_{z_n}(p) \sin \angle_p( \ga(n), \supp {z_n})\\
			&\leq \theta_0\sum\limits_{p\in \ga(n) \cap \supp {z_n}} w_{z_n}(p) \\
			&\leq l_{\ga (n)}(m)\epsilon.
		\end{align*}
		Hence we get that the closed geodesics $\{\ga(n)\}_n$ do not get uniformly lengthened by the strip map. Hence, $f(z_n)$ converges to a point in $\partial \adm m$. 
		
		Now we show that $f(z_n)\not\to \vec 0$. Let $\lambda:=\lim\limits_{n\to\infty} \supp{z_n}$ be the limit in Hausdorff topology. The normalisation condition states that for every $n\in \N$, we have
		\begin{equation*}
			\sum_{p\in \partial \sh\cap\supp{x_n} }w_{z_n}(p)=1.
		\end{equation*}
		So for every $p\in \partial \sh\cap\supp{x_n} $, we have $w_{z_n}(p)\geq \frac{1}{2N}$. 
		Let $b$ be a boundary component of the surface such that for every $n\in \N$ it contains an endpoint $p(n)$ of an arc $\al_n$ in $z_n$. Then 
		\begin{align*}
			\mathrm{d}l_b(f(z_n))&\geq \frac{\sin \angle_{p(n)}(b,\supp{z_n})}{N}\geq \frac{\sin \theta_0}{N}>0.
		\end{align*}
		Thus we have that $f(z_n)$ is bounded away from $\vec 0$.
	\end{enumerate}
\end{proof}

\section{Parametrisation of Margulis spacetimes}\label{dmst}
In this section we first recall the parametrisation of Margulis spacetimes using the pruned arc complex and the construction of the fundamental domain of a Margulis spacetime from an admissible deformation of a compact hyperbolic surface with boundary, as done in \cite{dgk}. 

\subsection{Undecorated Margulis spacetimes}
\paragraph{Drumm's construction of proper cocycles.} 
Let $\rho:\Ga\longrightarrow G$ be a convex cocompact representation. A fundamental domain for the action of $\rho(\Ga)$ on the hyperbolic plane $\HP$ is bounded by finitely many pairwise disjoint geodesics. These geodesics are used to construct the stems of pairwise disjoint crooked planes in $\Min$. Then these planes are made disjoint from each other by adding points from their respective stem quadrants. The polyhedron bounded by these new crooked planes is a fundamental domain for the action of the group $\Gamma$ and  the resulting manifold $X/\Gamma$ is complete. Finally, Drumm determined $u$.

\paragraph{From proper cocycles to Margulis spacetimes, \cite{dgk}.}
Let $S_c$ be a compact hyperbolic surface with totally geodesic boundary. Let $\rho:\fg{S_c}\longrightarrow \pgl$ be a holonomy representation and $u:\fg{S_c}\longrightarrow \lalg$ be a $\rho$-cocycle such that $[u]\in \adm{[\rho]}$. From Theorem 1.7 in \cite{dgk}, we know that the projectivised strip map when restricted to the pruned arc complex of the surface $S_c$ is a homeomorphism onto its image $\adm {[\rho]}$.
So there exists a point $x\in\sac {S_c}$ and a unique simplex $\sigma$ such that $\mathbb{P}f(x)=[u]\in \mathbb{P}^+ \adm {[\rho]}$ and $x\in \inte\sigma$. So $x=\sum_i t_i [\al_i]$ with $\sum_i t_i =1$ and $f(x)=\sum_i t_i f_{\al_i}(m)$. Corresponding to this linear combination of strip maps, we get a class of tile maps $\phi:\ltile\longrightarrow \lalg$ that are $(\rho(\fg{S_c}),u)$-equivariant. Let $\al\in \ed$ be any arc of $\sigma$ and $\wt \al$ be any lift. There exists tiles $d_1,d_2\in\ltile$ that have $\wt \al$ as their common internal edge. Suppose that the geodesic arc $\wt\al$ is positively transversely oriented from $d_1$ to $d_2$. Then the Killing field $\phi(d_2)-\phi(d_1)$ is hyperbolic and represents the term $t_\al f_{\al}(m)$ in $f(x)$. Let $\mathbf v_{\wt\al}\in \du{\al}$ be a hyperbolic Killing field with attracting and repelling fixed points given by $[\mathbf {v}_{\wt\al}^+],[\mathbf v_{\wt\al}^-]$ such that the triplet $(\mathbf {v}_{\wt\al}^+, \mathbf {v}_{\wt\al}, \mathbf {v}_{\wt\al}^-)$ is positively oriented and the tile $d_2$ lies to the left of the axis when viewed from $[\mathbf v_{\wt\al}^-]$. Then the crooked plane associated to $\wt \al$ is given by $\mathcal{P}_{\wt\al}:=\mathcal P(\mathbf {w}_{\wt\al},\mathbf {v}_{\wt\al})$, where $\mathbf {w}_{\wt\al}:=\frac{\phi(d_1)+\phi(d_2)}{2}$. For other arcs in the orbit of $\wt\al$, the crooked plane is defined as: for every $\ga\in \fg{S_c}$, $\mathcal{P}_{\rho(\ga)\cdot\wt\al}= \rho(\ga)\cdot \mathcal{P}_{\wt\al} +u(\ga).$ 

Firstly, it is shown that for every two disjoint arcs $\alo,\alt\in \ed$ their associated crooked planes $\mathcal{P}_{\rho(\ga)\cdot\alo}, \mathcal{P}_{\rho(\ga)\cdot\alt}$ are disjoint by using Drumm's sufficient condition. Then they consider a fundamental domain of the surface bounded by finitely many arcs in $\led$ and show that the associated crooked planes form a fundamental domain for the Margulis spacetime. We shall adapt this method to our surfaces with decorated spikes.

\subsection{Decorating a Margulis spacetime}
\subsubsection{Photons and Killing fields}
Consider the projective disk model of $\HP$ and a point $p\in \HPb$. Recall that an open horoball $h$ based at $p$ is the projective image of the subset $H(\mathbf v)=\{\mathbf w\in \HP\mid \bil{\mathbf w}{\mathbf{v}}>-1 \}$ of the hyperboloid $\HP$, where $\mathbf v$ is a future-pointing light-like point in $\pinv p$. If $k>k'>0$, then the horoball $h:=\mathbb{P}H(k\mathbf v_0)$ is smaller than the horoball $h':=\mathbb{P}H(k'\mathbf v_0)$. 

\begin{definition}
	Let $\mathbf {v_0}\in \Min$ be a future-pointing light-like vector and let $\mathbf v\in\Min$ be any point. Then the affine line $\mathcal L(\mathbf v,\mathbf{v_0}):=\mathbf v+\R \mathbf {v_0}$ is called a \emph{photon}. 
\end{definition}
A vector $\mathbf u\in \mathcal L(\mathbf v,\mathbf{v_0})$ corresponds to a Killing field that moves the vector $\mathbf{v_0}$ in the direction $\mathbf u \mcp \mathbf{v_0}$. A vector $\mathbf w\in H$ is moved in the direction $\mathbf u\mcp \mathbf w$.
\begin{figure}
	\centering
	\includegraphics[height=5cm]{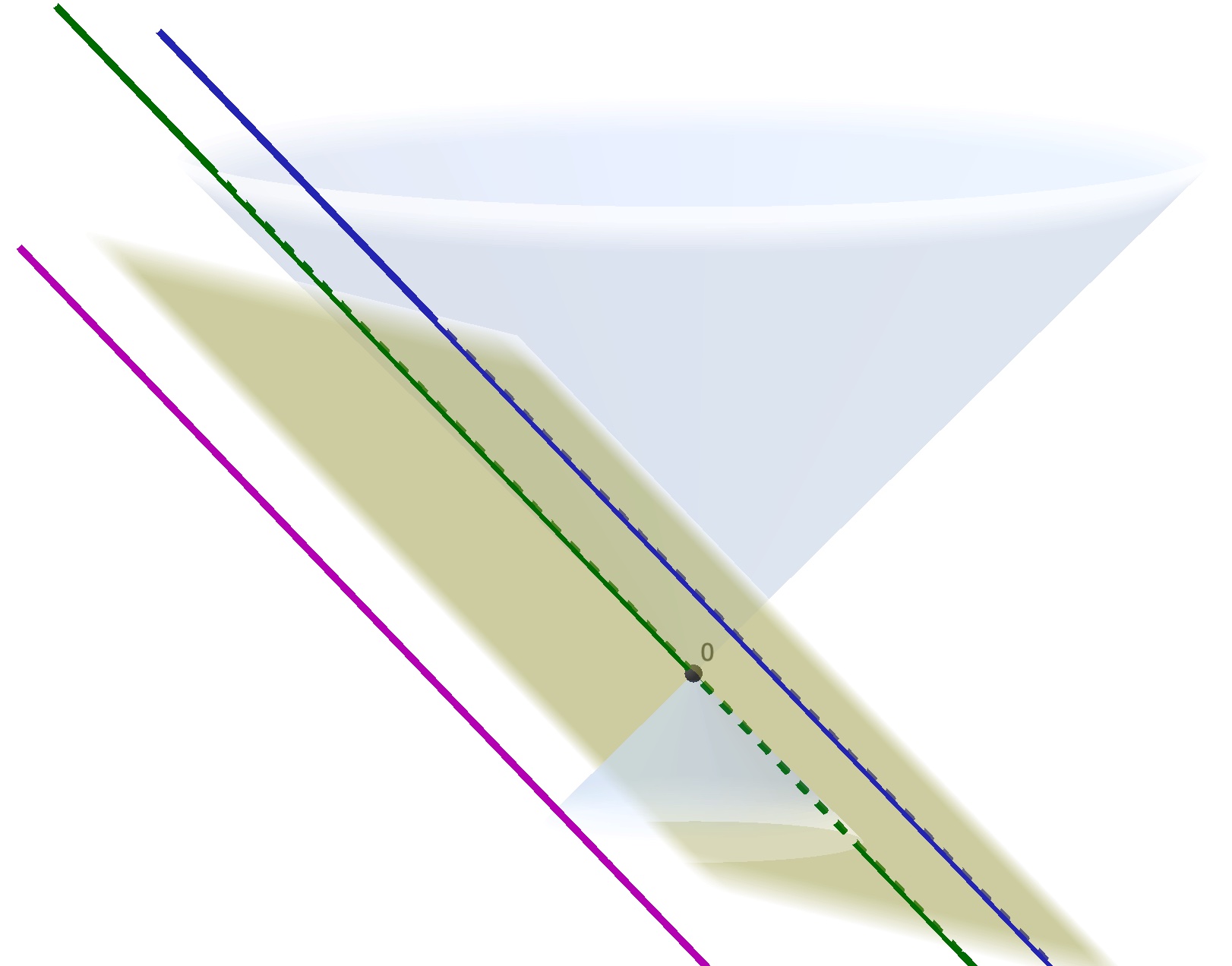}
	\caption{The different types of photons.}
	\label{typesofphoton}
\end{figure}
\begin{itemize}
	\item When $\mathbf v\in \R \mathbf{ v_0}$, the photon $\mathcal L(\mathbf v,\mathbf{v_0})$ is the vectorial line $ \R \mathbf {v_0}$, coloured green in Fig.\ref{typesofphoton}. Its non-zero points correspond to parabolic Killing fields that fix the ideal point $[\mathbf {v_0}]$ in the hyperbolic plane and preserve the horoballs based at this ideal point as sets: for $k\in \R\smallsetminus\{0\}$, $k\mathbf{v_0}\mcp \mathbf{v_0}=0$. So the vector $\mathbf{v_0}$ and hence the set $H(\mathbf v)$ is preserved by the flow of the Killing field associated to $kv_0$. 
	
	\item When $\mathbf v$ is contained in the light-like plane $\du{\mathbf{v_0}}$, the photon also lies inside $\du{\mathbf{v_0}}$. Such a photon is coloured blue in Fig.\ref{typesofphoton}. Any vector $\mathbf u$ on such a photon, that is not contained in $\R\mathbf{v_0}$, is a hyperbolic Killing field with one of its fixed points at $[\mathbf{v_0}]$. We have that $\mathbf u\mcp \mathbf{v_0}\in \du {\mathbf u}\cap\du {\mathbf{v_0}}=\R \mathbf{v_0}$. So the vector $\mathbf{v_0}$ and the set $H$ gets scaled by the flow of the Killing vector field $\mathbf u$. The connected component of the set $\du{\mathbf{v_0}}\smallsetminus\R\mathbf{v_0}$ that contains the hyperbolic Killing fields whose attracting (resp.\ repelling) fixed point is given by $[\mathbf{ v_0}]$ shrinks (resp.\ enlarges) the horoballs centered at this point.
	\item When $\mathbf v\in \Min\smallsetminus\du{\mathbf{v_0}}$, any vector $\mathbf u=\mathbf v+k\mathbf{v_0}\in\mathcal L(\mathbf v,\mathbf{v_0})$ moves the light-like vector away from $\R\mathbf{v_0}$ and in the direction given by $\mathbf u\mcp \mathbf{ v_0}$. Such a photon is coloured in pink in Fig.\ref{typesofphoton}. When $\mathbf v$ lies above (resp.\ below), the point $[\mathbf{ v_0}]$ is moved in the clockwise (resp.\ anticlockwise) direction on $\HPb$. 
\end{itemize}
The space of photons can be identified with the tangent bundle over the space of horoballs, modulo simultaneous scaling of all horoballs.
This identification is equivariant for the actions of $G \ltimes \lalg = T(G)$.

\subsubsection{Handedness}\label{hand}
\begin{figure}
	\centering
	\includegraphics[width=\linewidth]{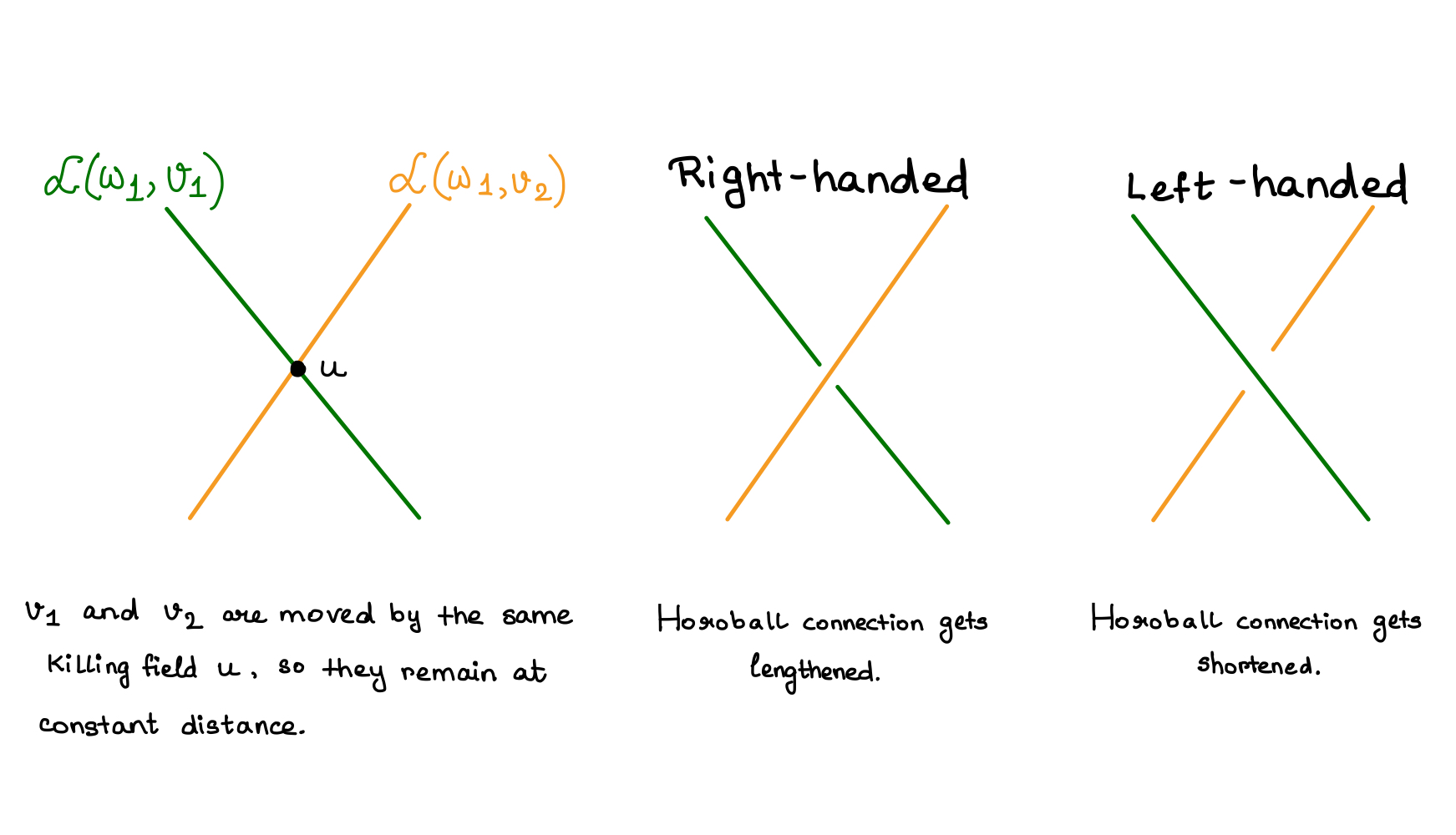}
	\caption{A pair of photons.}
	\label{pairofphotons}
\end{figure}
Let $\mathbf v_1,\mathbf v_2\in \Min$ be two future-pointing light-like vectors. For $i=1,2$, let $\mathbf w_i\in \mathcal{W}^+(\mathbf v_i)$, where $\mathcal{W}^+(\mathbf v_i)$ is the positive wing of $\mathbf v_i$. Then the photon $\mathcal{L}(\mathbf{ w}_i,\mathbf{ v}_i)$ consists of hyperbolic Killing fields that have $[\mathbf v_i]$ as attracting fixed point. So for every $i=1,2$, the vector $\mathbf v_i$ gets infinitesimally deformed towards $k\mathbf v_i$ for $k>1$ and the horoball $h_i:=\mathbb P (H(\mathbf v_i))$ gets shrunken. Finally, consider the pair of photons $\{\mathcal{L}(\mathbf{ w}_1,\mathbf{ v}_1), \mathcal{L}(\mathbf{ w}_2,\mathbf{ v}_2)\}$. Any horoball connection joining the decorated spikes $([\mathbf v_1],h_1)$ and $([\mathbf v_2],h_2)$ gets lengthened.

Now let $\{\mathcal{L}(\mathbf{ w}_1,\mathbf{ v}_1), \mathcal{L}(\mathbf{ w}_2,\mathbf{ v}_2)\}$ be any pair of disjoint photons. They are contained in the two affine light-like planes $A_1:=\mathbf w_1+\du{\mathbf v_1},A_2:=\mathbf w_2+\du{\mathbf v_2}$, respectively.
Let $\mathbf v\in \mathcal{L}(\mathbf{w}_1,\mathbf{v}_1)\cap A_2$ 
and $\mathbf v'\in\mathcal{L}(\mathbf{w}_2,\mathbf{v}_2)\cap A_1$. 
Then the relative motions of $\mathbf{v_1}$ and $\mathbf{v_2}$ is given by
\begin{align*}
	\mathbf v-\mathbf v'&=\mathbf w_1-\mathbf w_2 - \frac{\bil{\mathbf w_1}{\mathbf v_2}}{{\bil {\mathbf v_1}{\mathbf v_2}}} \mathbf v_1+\frac{\bil{\mathbf w_2}{\mathbf v_1}}{{\bil {\mathbf v_1}{\mathbf v_2}}}\mathbf v_2\\
	&=\frac{ \bil{\mathbf w_1-\mathbf w_2}{\mathbf v_1 \mcp \mathbf v_2}}{\norm {\mathbf v_1 \mcp \mathbf v_2}^2} (\mathbf v_1\mcp \mathbf v_2).
\end{align*}
The sign of the real number $\bil{\mathbf w_1-\mathbf w_2}{\mathbf v_1 \mcp \mathbf v_2}$ gives the \emph{handedness} of the pair  $\{\mathcal{L}(\mathbf{ w}_1,\mathbf{ v}_1), \mathcal{L}(\mathbf{ w}_2,\mathbf{ v}_2)\}$.

\subsection{From decorated surfaces to decorated Margulis spacetimes}
In this section, we will adapt the parametrisation of Margulis spacetimes to our case of hyperbolic surfaces with decorated spikes. We start by defining decorated Margulis spacetimes.

Let $\sh$ be a hyperbolic surface with $Q$ decorated spikes, endowed with a decorated metric $m=[\rho,\mathbf{x},\mathbf h]$ in $\tei \sh$. Then the metric on the convex core $S_c$ is given by $[\rho]$. The admissible cone $\adm m$ is an affine bundle over the admissible cone $\adm{[\rho]}$ of the convex core; denote by $\pi$, the bundle projection $\pi:\adm m \longrightarrow \adm{[\rho]}$. The fibres are open subsets of $\R^{2Q}$ that are stable under the scaling of horoballs.  

Let $[u]\in \adm m$ be an admissible deformation of the surface $\sh$. Let $[u_0]:=\pi([u])$. Then $u_0$ is a proper $\rho$-cocycle and the group of isometries $\Gamma^{(\rho,u_0)}$ acts properly discontinuously on $\Min$. The quotient $M:=\Min/\Gamma^{(\rho,u_0)}$ is a Margulis spacetime, which we decorate with photons in the following way: 
the infinitesimal deformation $[u]$ imparts motion to every lift of each decorated spike of the surface. From the previous section, we know that set of Killing fields realising this particular variation to an ideal point decorated with a horoball, happens to be a photon. This collection of photons, denoted by $\mathscr L$, is $\Gamma^{(\rho,u_0)}$-equivariant and is the decoration of the underlying Margulis spacetime. The pair $(M, \mathscr L)$ is called a \emph{decorated} Margulis spacetime. 

Next we will give another way of looking at this decoration using tile maps. 
We know that the projectivised strip map when restricted to the pruned arc complex is a homeomorphism onto its image $\adm m$. So there exists a point $x=\sac \sh$ and a unique big simplex $\sigma$ such that $\mathbb{P}f(x)=[u]\in \mathbb{P}^+ \adm m$ and $x\in \inte\sigma$. So $x=\sum_i t_i [\al_i]$ with $\sum_i t_i =1$, $t_i>0$ for every $i$ and $f(x)=\sum_i t_i f_{\al_i}(m)$. Corresponding to this linear combination of strip maps we get a class of tile maps $\phi:\ltile\longrightarrow \lalg$. Now suppose that the surface has $Q$ spikes and write the spike vector $\mathbf x$ as $(x_1,\ldots,x_Q)$. Since the arcs of $\sigma$ decompose the surface into tiles with at most one spike, there exist exactly $Q$ tiles $d_1,\ldots,d_Q$ such that $x_i\in d_i$ for every $i=1,\ldots, Q$. Using the tile map, we get a collection of $Q$ Killing fields $\phi(d_1),\ldots, \phi(d_Q)$ where $\phi(d_i)$ acts on the ideal point $x_i$. Now suppose that $\mathbf h=(h_1,\ldots,h_Q)$ is the horoball decoration given by the metric $m$. Then for each $i=1,\ldots,Q$, there exists a future pointing light-like vector $\mathbf {v}_i$ and the set $H(\mathbf {v}_i)$ such that $x_i=[\mathbf{v}_i]$ and $h_i=[H(\mathbf{v}_i)]$. Then consider the collection of photons of the form $\phi(d_i)+ \R \mathbf{v}_i$ for $i=1, \ldots, Q$ and take their $\Gamma^{(\rho,u_0)}$-orbit. 

\begin{remark}
	Note that these photons are pairwise disjoint. If two photons intersect, their intersection point is a Killing field that realizes the motions of the two corresponding horoballs, hence the horoball connection has zero infinitesimal length variation. Hence the infinitesimal deformation $[u]$ fails to be admissible.
\end{remark}


\begin{remark}
	Every pair of photons has the same handedness, because every horoball connection is lengthened.
\end{remark}

\subsubsection{From decorated Margulis space-times to admissible deformations.}  Let $\Ga$ be a finitely generated free discrete group acting properly discontinuously on $\Min$ and its representation $\rho:\Ga\longrightarrow G\ltimes \lalg$. Let $(M:=\Min/\rho(\Ga), \mathscr L)$ be a decorated Margulis spacetime with convex cocompact linear part $\rho_0:\Ga\longrightarrow G$. Using Drumm's construction of proper cocycles, we have that $\Ga=\Gamma^{(\rho_0,u_0)}$, where $u_0$ is a proper $\rho_0$-cocycle. The surface $S_c:=\HP/\rho_0(S_c)$ is compact with totally geodesic boundary. Denote its boundary components by $b_1,\ldots,b_n$. 

The set $\mathscr L$ is $\Gamma^{(\rho_0,u_0)}$-equivariant; there exists finitely many pairs $(\mathbf{w}_1, \mathbf{v}_1),\ldots, (\mathbf{w}_Q, \mathbf{v}_Q)$ of points in $\Min$ such that for every $i$, the vector $\mathbf{v}_i$ is future-pointing and light-like and $\mathscr L$ is generated by the photons $\mathcal L_i:=\mathcal L(\mathbf w_i, \mathbf v_i)$, $i=1,\ldots, Q$. This gives us $Q$ ideal points $x_i=[\mathbf{v}_i]\in \HPb$. Take the $\rho_0(\Ga)$-orbit of this collection and join every consecutive pair, that lie on the same side of a lift of the boundary loop $b_i$, by a geodesic. Let $R$ be the simply-connected region in $\HP$ bounded these geodesics. Then we get a hyperbolic surface with decorated spikes $\sh:=R/\rho_0(\Ga)$ with the metric $m=[\rho_0,\mathbf{x}, \mathbf{h} ]$, where $\mathbf {x}=(x_1,\ldots, x_Q)$ and $\mathbf {h}=(h_1,\ldots,h_Q)$, $h_i= \mathbb P (H(\mathbf{ v}_i))$. The surface $S_c$ is the convex core of $\sh$.

The admissible deformation of $\sh$ is determined in the following way: for every $i=1,\ldots, Q$, the photon $\mathcal L_i$ imparts infinitesimal motion to the spike $x_i$ as well as the horoball $h_i$ in the sense that $\mathcal L_i$ is exactly the set of Killing fields all of whom cause $h_i$ to vary in a certain infinitesimal way. Since no two photons intersect, every horoball connection is deformed and since every pair of photons has the same handedness, every horoball connection gets lengthened. Thus we get an admissible deformation $[u]\in \adm m$ with $\pi([u])=[u_0]$. 

\subsubsection{From admissible deformations to decorated Margulis space-times.}
Let $\sh$ be a hyperbolic surface with $Q$ decorated spikes, endowed with a decorated metric $m$, which is of the form $m=[\rho,\mathbf{x},\mathbf h]\in \tei \sh$. Let $[u]\in \adm m$ be an admissible deformation of the surface $\sh$. Let $[u_0]:=\pi([u])$. Then $u_0$ is a proper $\rho$-cocycle and the group of isometries $\Gamma^{(\rho,u_0)}$ acts properly discontinuously on $\Min$. By Theorem \ref{maindec} there exists a unique point $x=\sac \sh$ and a unique big simplex $\sigma$ such that $\mathbb{P}f(x)=[u]\in \mathbb{P}^+ \adm m$ and $x\in \inte\sigma$. So $x=\sum_i t_i [\al_i]$ with $\sum_i t_i =1$, $t_i>0$ for every $i$ and $f(x)=\sum_i t_i f_{\al_i}(m)$. Corresponding to this linear combination of strip maps, we get a class of tile maps $\phi:\ltile\longrightarrow \lalg$ that are $(\rho(\fg{S_c}),u)$-equivariant. Let $\al\in \ed$ be any arc of $\sigma$ and $\wt \al$ be any lift. There exists tiles $d_1,d_2\in\ltile$ that have $\wt \al$ as their common internal edge. The arc is either finite or joins a decorated spike with a bounary component.  Let $\phi(d_2)-\phi(d_1)$ be the Killing field that represents the term $t_\al f_{\al}(m)$ in $f(x)$. When $\al$ is finite, the difference is a hyperbolic Killing field belonging to the stem quadrant $\sq$. Otherwise, it is a parabolic Killing field fixing the ideal endpoint of $\wt\al$. Define the associated crooked plane as before: 
$\mathcal{P}_{\wt\al}:=\mathcal P(\mathbf {w}_{\wt\al},\mathbf {v}_{\wt\al})$, with $\mathbf {w}_{\wt\al}:=\frac{\phi(d_1)+\phi(d_2)}{2}$. Let $R$ be a fundamental domain of the surface $\sh$ bounded by some arcs $e_1,f_1,\ldots, e_k,f_k$ in $\led$ such that there exists $\ga_1,\cdots,\ga_k\in \fg{\sh}$ such that for $i=1,\cdots,k$, $f_i=\rho(\ga_i)\cdot e_i$. Since there are no parabolic elements in $\rho(\fg{\sh})$, nor any spiralling arcs, for every pair $(e_i,f_i)$ of spike-to-edge arcs, the spikes are distinct. So there exists an edge-to-edge arc $\al\in \ed$ whose lift $\wt\al$ separates $e_i$ from $f_i$. Since the arc $\wt\al$ is disjoint to both $e_i$ and $f_i$ in $\cHP$, its associated crooked plane $\mathcal{P}(\mathbf{w}_{\wt\al}, \mathbf{v}_{\wt\al})$ separates the crooked planes $\mathcal P_{e_i}, \mathcal P_{f_i}$. Hence we have that for every $i=1,\ldots, k$, the crooked planes $\mathcal P_{e_i}, \mathcal P_{f_i}$ are disjoint and $(\rho(\ga_i), u_0(\ga_i))\mathcal P_{e_i}=\mathcal P_{f_i}$. The region $\mathcal D$ bounded by these crooked planes is a fundamental domain for the action of $\Gamma^{(\rho,u_0)}$ on $\Min$. 

Thus we have proved the following theorem:
\begin{theorem}
	Let $\sh$ be a hyperbolic surface with decorated spikes and let $\rho:\fg {\sh} \rightarrow \pgl$ be a holonomy representation.
	Let $\mst$ be the space of all Margulis spacetimes with convex cocompact linear part as $\rho$. Then there is a bijection $\Psi: \sac\sh \rightarrow\mst$.
\end{theorem}

\begin{acknowledgements}
This work was done during my PhD at Universit\'e de Lille from 2017-2020 funded by the AMX scholarship of Ecole Polytechnique. I would like to thank my thesis advisor Fran\c{c}ois Gu\'eritaud for his valuable guidance and extraordinary patience. I am grateful to my thesis referees Hugo Parlier and Virginie Charette for their helpful remarks. I am also grateful to University of Luxembourg for funding my postdoctoral research (Luxembourg National Research Fund OPEN grant O19/13865598).
\end{acknowledgements}

\bibliography{DecoMarg.bib}
\bibliographystyle{spmpsci}
\end{document}